\documentclass[12pt,english]{article}
\usepackage[T1]{fontenc}
\usepackage[latin9]{inputenc}
\usepackage{geometry}
\usepackage{babel}
\usepackage{mathtools}
\usepackage{amsmath}
\usepackage{amsthm}
\usepackage{amssymb}
\usepackage[unicode=true]
 {hyperref}
\usepackage{breakurl}

\makeatletter

\providecommand{\tabularnewline}{\\}

\theoremstyle{plain}
\newtheorem{thm}{\protect\theoremname}
  \theoremstyle{plain}
  \newtheorem{cor}[thm]{\protect\corollaryname}
  \theoremstyle{plain}
  \newtheorem{prop}[thm]{\protect\propositionname}

\usepackage{bm}
\usepackage{holtpolt}
\usepackage{microtype}

\usepackage{tikz}
\usetikzlibrary{arrows}

\usepackage{colortbl}
\usepackage{xparse}
\usepackage{caption}

\DeclareMathOperator{\Mat}{Mat}
\DeclareMathOperator{\occ}{occ}
\DeclareMathOperator{\mk}{mk}
\DeclareMathOperator{\asc}{asc}
\DeclareMathOperator{\plt}{plt}
\DeclareMathOperator{\Fib}{Fib}
\DeclareMathOperator{\pk}{pk}
\DeclareMathOperator{\val}{val}

\usepackage{titlesec}
\titleformat{\section}
  {\normalfont\large\bfseries}{\thesection.}{.8ex}{} 
\titleformat{\subsection}
  {\normalfont\bfseries}{\thesubsection.}{.8ex}{} 

\usetikzlibrary{shapes}

\tikzstyle{pathdefault}=[draw, line width=1, solid, color=black]
\tikzstyle{nodedefault}=[circle, inner sep=1.5, fill=black]
\tikzstyle{empty}=[]
\tikzstyle{nodeellipsis}=[circle, inner sep=0.5, fill=black]
\tikzstyle{pathcolor1}=[draw, line width=1.3, densely dashed, color=red]
\tikzstyle{pathcolor2}=[draw, line width=1.6, densely dotted, color=blue]
\tikzstyle{pathcolorlight}=[draw, line width=1, dotted, color=lightgray]

\tikzstyle{arbpathcolor0}=[line width=1, dashdotted, color=black]
\tikzstyle{arbpathcolor1}=[line width=1, densely dashed, color=red]
\tikzstyle{arbpathdefault}=[line width=1, densely dotted, color=blue]

\newcounter{id}
\newcommand{\drawlinedotswithstyle}[4]{
 \def\x{{#3}}
 \def\y{{#4}}
 \tikzstyle{thispathstyle}=[#1]
 \tikzstyle{thisnodestyle}=[#2]
 \setcounter{id}{-1} 
 \foreach \j in {#3}{\stepcounter{id}} 
 \foreach \i in {1,...,\the\value{id}}{  
  \path[thispathstyle] (\x[\i],\y[\i]) --(\x[\i-1],\y[\i-1]); 
 }
 \foreach \i in {1,...,\the\value{id}}{  
  \node[thisnodestyle] at (\x[\i],\y[\i]) {}; 
 }
 \node[thisnodestyle] at (\x[0],\y[0]) {}; 
}

\DeclareDocumentCommand{\drawlinedots}{ O{pathdefault} O{nodedefault} m m}{\drawlinedotswithstyle{#1}{#2}{#3}{#4}}

\let\originalleft\left
\let\originalright\right
\renewcommand{\left}{\mathopen{}\mathclose\bgroup\originalleft}
\renewcommand{\right}{\aftergroup\egroup\originalright}

\makeatother

  \providecommand{\corollaryname}{Corollary}
  \providecommand{\propositionname}{Proposition}
\providecommand{\theoremname}{Theorem}

\begin{document}
\global\long\def\mk{\operatorname{mk}}
\global\long\def\occ{\operatorname{occ}}
\global\long\def\Mat{\operatorname{Mat}}
\global\long\def\asc{\operatorname{asc}}
\global\long\def\plt{\operatorname{plt}}
\global\long\def\Fib{\operatorname{Fib}}
\global\long\def\pk{\operatorname{pk}}
\global\long\def\val{\operatorname{val}}

\title{A generalized Goulden\textendash Jackson cluster method and lattice
path enumeration}

\author{Yan Zhuang\\
Department of Mathematics\\
Brandeis University\\
\texttt{\href{mailto:zhuangy@brandeis.edu}{zhuangy@brandeis.edu}}}
\maketitle
\begin{abstract}
The Goulden\textendash Jackson cluster method is a powerful tool for
obtaining generating functions counting words in a free monoid by
occurrences of a set of subwords. We introduce a generalization of
the cluster method for monoid networks, which generalize the combinatorial
framework of free monoids. As a sample application of the generalized
cluster method, we compute bivariate and multivariate generating functions
counting Motzkin paths\textemdash both with height bounded and unbounded\textemdash by
statistics corresponding to the number of occurrences of various subwords,
yielding both closed-form and continued fraction formulas.
\end{abstract}
\textbf{\small{}Keywords:}{\small{} Goulden\textendash Jackson cluster
method, free monoids, lattice paths, Motzkin paths, generating functions,
statistics}{\let\thefootnote\relax}

\section{Introduction}

Given a finite or countably infinite set $A$, let $A^{*}$ be the
set of all finite sequences of elements of $A$, including the empty
sequence. We call $A$ an \textit{alphabet}, the elements of $A$
\textit{letters}, and the elements of $A^{*}$ \textit{words}. By
defining an associative binary operation on two words by concatenating
them, we see that $A^{*}$ is a monoid under the operation of concatenation
(where the empty word is the identity element), and we call $A^{*}$
the \textit{free monoid} on $A$. The length $l(\alpha)$ of a word
$\alpha\in A^{*}$ is the number of letters in $\alpha$. For $\alpha,\beta\in A^{*}$,
we say that $\beta$ is a \textit{subword} of $\alpha$ if $\alpha=\gamma_{1}\beta\gamma_{2}$
for some $\gamma_{1},\gamma_{2}\in A^{*}$, and in this case we also
say $\alpha$ \textit{contains} $\beta$.

More generally, a \textit{free monoid} is a monoid isomorphic to a
free monoid on some alphabet. The combinatorial framework of free
monoids is useful for the study of combinatorial objects that can
be uniquely decomposed into sequences of ``prime elements'', corresponding
to letters in an alphabet. This framework can furthermore be generalized
using what are called ``monoid networks'', which were first introduced
by Gessel \cite[Chapter 6]{gessel-thesis} in a slightly different
yet equivalent form called ``$G$-systems''.\footnote{The term ``$G$-system'' was dropped at the request of Ira Gessel,
who prefers the name ``monoid network'' given by the author.} Roughly speaking, a monoid network consists of a digraph $G$ with
each arc assigned a set of letters from an alphabet $A$, in which
the set of sequences of arcs in $G$ is given a monoid structure and
is equipped with a monoid homomorphism.

The Goulden\textendash Jackson cluster method allows one to determine
the generating function for words in a free monoid $A^{*}$ by occurrences
of words in a set $B\subseteq A^{*}$ as subwords in terms of the
generating function for what are called ``clusters'' formed by words
in $B$, which is easier to compute. As its name suggests, this celebrated
result was first given by Goulden and Jackson in \cite{Goulden1979}.
The cluster method has seen a number of extensions and generalizations
\cite{Bassino2012,Edlin2000,Kong2005,Kupin2010,Noonan1999,Wen2005,Zeilberger2002},
and the cluster method itself can be viewed as a generalization of
the Carlitz\textendash Scoville\textendash Vaughan theorem, which
allows one to count words in a free monoid avoiding a specified set
of length 2 subwords.

In this paper, we give a new generalization of the Goulden\textendash Jackson
cluster method of a different flavor: we generalize the cluster method
to monoid networks, which gives a way of counting words in $A^{*}$
corresponding to walks between two specified vertices in $G$ (that
is, words in a regular language if the alphabet $A$ is finite) by
occurrences of subwords in a set $B$. Then the original version of
the cluster method corresponds to the special case in which $G$ consists
of a single vertex with a loop to which the entire alphabet $A$ is
assigned.

The organization of this paper is as follows. In Section 2, we give
an expository account of the original Goulden\textendash Jackson cluster
method. In Section 3, we introduce the combinatorial framework of
monoid networks and present our generalization of the cluster method
for monoid networks. Finally, in Section 4, we demonstrate how our
monoid network version of the cluster method can be used to tackle
problems in lattice path enumeration.

Although many types of lattice paths can be represented as walks in
certain digraphs, in this paper we focus on Motzkin paths, which are
paths in $\mathbb{Z}$ beginning and ending at 0 with steps $-1$,
$0$, and $1$ (also called ``down steps'', ``flat steps'', and ``up
steps'', respectively). We consider both regular Motzkin paths and
Motzkin paths bounded by height, and our results include bivariate
and multivariate generating functions counting these paths by ascents,
plateaus, peaks, and valleys\textemdash all of which are statistics
that are determined by occurrences of various subwords in the underlying
word of the Motzkin path\textemdash as well as generating functions
for Motzkin paths with restrictions on the heights at which these
subwords can occur, yielding both closed-form and continued fraction
formulas. Several interesting identities are uncovered along the way.

\section{The Goulden\textendash Jackson cluster method}

We begin this section with a motivating problem: let $A$ be a finite
or countably infinite alphabet and suppose that we want to count words
in $A^{*}$ that do not contain a specified set $B$ of forbidden
subwords of length at least 2. The Goulden\textendash Jackson cluster
method allows us to count this restricted set of words by counting
``clusters'' formed by words in $B$, which we shall define shortly. 

Given a word $\alpha=a_{1}a_{2}\cdots a_{n}\in A^{*}$ (where the
$a_{i}$ are letters) and a set $B\subseteq A^{*}$, we say that $(i,\beta)$
is a \textit{marked subword} of $\alpha$ if $\beta\in B$ and 
\[
\beta=a_{i}a_{i+1}\cdots a_{i+l(\beta)-1},
\]
that is, $\beta$ is a subword of $\alpha$ starting at position $i$.
Moreover, we say that $(\alpha,S)$ is a \textit{marked word} on $\alpha$
if $\alpha\in A^{*}$ and $S$ is any set of marked subwords of $\alpha$.

For example, suppose that $A=\{a,b,c\}$ and $B=\{abc,bca\}$. Then
\begin{equation}
\{abcabbcabc,\{(1,abc),(2,bca),(6,bca)\}\},\label{e-mkwrd}
\end{equation}
is a marked word which can also be displayed as

\begin{center}
\begin{tikzpicture}

\node at (0,0) {$a\:b\:c\:a\:b\:b\:c\:a\:b\:c\:.$};

\draw[red] (-1.08,-0.02) ellipse (13bp and 8bp);
\draw[red] (-0.78,-0.02) ellipse (13bp and 8bp);
\draw[red] (0.35,-0.02) ellipse (13bp and 8bp);

\end{tikzpicture}
\end{center}

The concatenation of two marked words is defined in the obvious way.
For example, (\ref{e-mkwrd}) can be obtained by concatenating $\{abca,\{(1,abc),(2,bca)\}\}$
and $\{bbcabc,\{(2,bca)\}\}$, i.e., \begin{center}
\begin{tikzpicture}

\node at (0,0) {$a\:b\:c\:a\quad$ and $\quad b\:b\:c\:a\:b\:c\:.$};

\draw[red] (-1.91,-0.02) ellipse (13bp and 8bp);
\draw[red] (-1.62,-0.02) ellipse (13bp and 8bp);
\draw[red] (1.18,-0.02) ellipse (13bp and 8bp);

\end{tikzpicture}
\end{center}

A marked word on $\alpha$ is called a \textit{cluster} on $\alpha$
if it is not a concatenation of two nonempty marked words. So, (\ref{e-mkwrd})
is not a cluster, but

\begin{center}
\begin{tikzpicture}

\node at (0,0) {$b\:c\:a\:b\:c\:a$};

\draw[red] (-0.42,-0.03) ellipse (13bp and 8bp);
\draw[red] (0.13,-0.03) ellipse (13bp and 8bp);
\draw[red] (0.41,-0.03) ellipse (13bp and 8bp);

\end{tikzpicture}
\end{center}

\noindent is a cluster. Two additional examples of clusters, using
$A=\{a\}$ and $B=\{aaaa\}$, are

\noindent \begin{center}
\begin{center}
\begin{tikzpicture}

\node at (0,0) {$a\:a\:a\:a\:a\:a$};

\draw[red] (-0.31,0) ellipse (18bp and 8bp);
\draw[red] (0.32,0) ellipse (18bp and 8bp);

\end{tikzpicture}
\end{center}
\par\end{center}

\noindent and

\noindent \begin{center}
\begin{center}
\begin{tikzpicture}

\node at (0,0) {$a\:a\:a\:a\:a\:a\:,$};

\draw[red] (-0.41,0.04) ellipse (18bp and 8bp);
\draw[red] (-0.09,0.04) ellipse (18bp and 8bp);
\draw[red] (0.22,0.04) ellipse (18bp and 8bp);

\end{tikzpicture}
\end{center}
\par\end{center}

\noindent which we include to emphasize the fact that a cluster is
not required to be ``maximal'' in the sense that every possible marked
subword must be included. If a word $\alpha$ has only one possible
cluster, then there is no need to indicate the positions of the marked
subwords and we say by abuse of language that the only cluster on
$\alpha$ is itself.

Before formally presenting the cluster method, we introduce some additional
notation. For a word $\alpha\in A^{*}$, let $\occ(\alpha)$ be the
number of occurrences in $\alpha$ of words in $B$ and let $C_{\alpha}$
be the set of all clusters on the word $\alpha$. Given a cluster
$c$, let $\mk(c)$ be the number of marked subwords in $c$. Given
a variable $t$ that commutes with all of the letters in $A$, define
\[
F(t)\coloneqq\sum_{\alpha\in A^{*}}\alpha t^{\occ(\alpha)}
\]
and 
\[
L(t)\coloneqq\sum_{\alpha\in A^{*}}\alpha\sum_{c\in C_{\alpha}}t^{\mk(c)},
\]
so that $F(t)$ is the generating function for words in $A^{*}$ by
the number of occurrences of words in $B$, and $L(t)$ is the generating
function for clusters by the number of marked subwords. Both $F(t)$
and $L(t)$ are elements of the formal power series algebra $K\langle\langle A^{*}\rangle\rangle[[t]]$,
where $K$ is a field of characteristic zero (which we can take to
be $\mathbb{C}$) and $K\langle\langle A^{*}\rangle\rangle$\textemdash called
the \textit{total algebra} of $A^{*}$ over $K$\textemdash is the
algebra of formal sums of words in $A^{*}$ with coefficients in $K$.
\begin{thm}[Goulden\textendash Jackson cluster method, version 1]
\label{t-gjcm1} Let $A$ be an alphabet and let $B\subseteq A^{*}$
be a set of words of length at least 2. Then, 
\[
F(t)=\bigg(1-\sum_{a\in A}a-L(t-1)\bigg)^{-1}.
\]
\end{thm}
\begin{proof}
We prove the equivalent statement 
\[
F(1+t)=\bigg(1-\sum_{a\in A}a-L(t)\bigg)^{-1}.
\]
We have 
\begin{align}
F(1+t) & =\sum_{\alpha\in A^{*}}\alpha(1+t)^{\occ(\alpha)}\nonumber \\
 & =\sum_{\alpha\in A^{*}}\alpha\sum_{k=0}^{\infty}{\occ(\alpha) \choose k}t^{k}\nonumber \\
 & =\sum_{\alpha\in A^{*}}\alpha\sum_{S\subseteq B_{\alpha}}t^{\left|S\right|},\label{e-gjcm1}
\end{align}
where $B_{\alpha}$ is the set of occurrences of words in $B$ in
$\alpha$. Note that (\ref{e-gjcm1}) counts marked words weighted
by the number of marked subwords that it contains, and from here it
is easy to see that 
\begin{align*}
F(1+t) & =\sum_{\alpha\in A^{*}}\alpha\sum_{S\subseteq B_{\alpha}}t^{\left|S\right|}\\
 & =\bigg(1-\sum_{a\in A}a-L(t)\bigg)^{-1}
\end{align*}
since every marked word is uniquely built from a sequence of letters
in $A$ and clusters.
\end{proof}
We indicate three specializations of Theorem \ref{t-gjcm1} that are
of particular importance:
\begin{itemize}
\item By setting $t=0$, we obtain 
\[
\bigg(1-\sum_{a\in A}a-L(-1)\bigg)^{-1}
\]
as the generating function for words in $A^{*}$ that do not contain
any words in $B$, which solves the problem posed at the beginning
of this subsection, assuming that we can compute the cluster generating
function $L(t)$.
\item If every word in $B$ has length exactly 2, then setting $t=0$ yields
a result which is sometimes called the Carlitz\textendash Scoville\textendash Vaughan
theorem, independently discovered by Fr\"oberg \cite[Section 4]{Froeberg1975},
by Carlitz, Scoville, and Vaughan \cite[Theorem 7.3]{Carlitz1976},
and by Gessel \cite[Theorem 4.1]{gessel-thesis}. In fact, Chapters
4 and 5 of Gessel's doctoral thesis \cite{gessel-thesis} are devoted
to the Carlitz\textendash Scoville\textendash Vaughan theorem and
its many enumerative applications.
\item By setting $t=1$, we obtain the free monoid identity 
\begin{equation}
\sum_{\alpha\in A^{*}}\alpha=\bigg(1-\sum_{a\in A}a\bigg)^{-1}.\label{e-fmtot}
\end{equation}
\end{itemize}
More generally, we can assign each word in $B$ its own variable.
Write $B=\{\beta_{1},\beta_{2},\dots\}$ so that the words in $B$
are ordered. (Here, $B$ is presented as countably infinite although
in most applications it is finite.) Given a word $\alpha\in A^{*}$,
let $\occ_{k}(\alpha)$ be the number of occurrences of $\beta_{k}$
in $\alpha$, and given a cluster $c$, let $\mk_{k}(c)$ be the number
of marked subwords in $c$ of the form $(j,\beta_{k})$ for some position
$j$. Let $t_{1},t_{2},\dots$ be variables that commute with each
other and with the letters of $A$, and define the generating functions
\[
F(t_{1},t_{2}\dots)\coloneqq\sum_{\alpha\in A^{*}}\alpha\prod_{k=1}^{\infty}t_{k}^{\occ_{k}(\alpha)}
\]
and 
\[
L(t_{1},t_{2},\dots)\coloneqq\sum_{\alpha\in A^{*}}\alpha\sum_{c\in C_{\alpha}}\prod_{k=1}^{\infty}t_{k}^{\mk_{k}(c)}.
\]
Then we have a refinement of Theorem \ref{t-gjcm1}, which follows
by the same reasoning as before.
\begin{thm}[Goulden\textendash Jackson cluster method, version 2]
\label{t-gjcm2} Let $A$ be an alphabet and let $B=\{\beta_{1},\beta_{2},\dots\}\subseteq A^{*}$
be a set of words of length at least 2. Then, 
\[
F(t_{1},t_{2}\dots)=\bigg(1-\sum_{a\in A}a-L(t_{1}-1,t_{2}-1,\dots)\bigg)^{-1}.
\]
\end{thm}
The statement of Theorem \ref{t-gjcm2} uses an infinite set $B$
and infinitely many variables $t_{i}$, but it is clear that the finite
case works as well. The number of variables also does not need to
equal the number of words in $B$; for example, we can have $B=\{\beta_{1},\dots,\beta_{k}\}$
along with two variables $t_{1}$ and $t_{2}$, and attach $t_{1}$
to all $\beta_{i}$ with $i$ odd and attach $t_{2}$ to all $\beta_{i}$
with $i$ even.

As an example, let $A=\{a,b,c\}$ and suppose that we want to count
words in $A^{*}$ by occurrences of $\beta_{1}=acb$ and $\beta_{2}=bc$.
Then the only clusters are $acb$, $bc$, and $acbc$, so 
\[
L(t_{1},t_{2})=acbt_{1}+bct_{2}+acbct_{1}t_{2}
\]
and by Theorem \ref{t-gjcm2}, we obtain 
\begin{equation}
F(t_{1},t_{2})=(1-a-b-c-acb(t_{1}-1)-bc(t_{2}-1)-acbc(t_{1}-1)(t_{2}-1))^{-1}\label{e-example1}
\end{equation}
as the generating function for words in $A^{*}$ by occurrences of
$acb$ and $bc$. By setting $t_{1}=t_{2}=0$, we obtain 
\begin{equation}
(1-a-b-c+acb+bc-acbc)^{-1}\label{e-example2}
\end{equation}
as the generating function for words in $A^{*}$ which contain neither
$acb$ nor $bc$.

Now, let $x$ be a variable that commutes with $t_{1}$ and $t_{2}$.
If we apply the homomorphism sending each of the letters to $x$,
we obtain the generating functions 
\[
\frac{1}{1-3x-x^{2}(t_{2}-1)-x^{3}(t_{1}-1)-x^{4}(t_{1}-1)(t_{2}-1)}
\]
and 
\[
\frac{1}{1-3x+x^{2}+x^{3}-x^{4}}
\]
from (\ref{e-example1}) and (\ref{e-example2}), respectively, where
$x$ is keeping track of the word length.

We say that the set $B$ is \textit{reduced} if no word $\beta\in B$
is a subword of another word $\beta^{\prime}$ in $B$. Although the
cluster method as presented above works regardless of whether $B$
is reduced, Goulden and Jackson gave a formula in their original paper
\cite{Goulden1979} for the cluster generating function when $A$
and $B$ are finite sets with $B$ reduced. A set $B$ of forbidden
subwords can always be replaced by a reduced set and still yield the
same restricted set of words; if $\beta\in B$ is a subword of $\beta^{\prime}\in B$,
then we can remove $\beta^{\prime}$ from $B$ because containing
$\beta^{\prime}$ implies containing $\beta$. However, the criterion
of having a reduced set can be an issue if we want to count words
by occurrences of subwords (that is, without setting $t=0$). For
instance, we would not be able to use Goulden and Jackson's formula
to compute the cluster generating function given $B=\{aba,abab\}$
since $aba$ is a subword of $abab$. 

As part of \cite{Noonan1999}, Noonan and Zeilberger wrote a Maple
package that handles the case where $B$ is arbitrary (i.e., not necessarily
reduced), but without a detailed explanation of their algorithms.
Bassino, Cl\'{e}ment, and Nicod\`{e}me \cite{Bassino2012} later
gave an explicit expression for the cluster generating function in
the non-reduced case. We omit these formulas of Goulden\textendash Jackson
and Bassino\textendash Cl\'{e}ment\textendash Nicod\`{e}me because
the cluster generating functions in Section 4 of this paper will require
essentially no computation.

\section{Our generalization of the cluster method}

\subsection{Monoid networks}

Throughout this section, fix a field $K$ of characteristic zero and
let $A$ be a finite or countably infinite alphabet. As in the previous
section, $K\langle\langle A^{*}\rangle\rangle$ is the total algebra
of $A^{*}$ over $K$. We also let $\Mat_{m}(K\langle\langle A^{*}\rangle\rangle)$
denote the algebra of $m\times m$ matrices with entries in $K\langle\langle A^{*}\rangle\rangle$.

Let $G$ be a digraph on the vertex set $[m]$ such that each arc
$(i,j)$ of $G$ is assigned a set of letters $P_{i,j}$ in $A$,
and let $P$ be the set of all pairs $(a,e)$ where $e=(i,j)$ is
an arc of $G$ and $a\in P_{i,j}$. Define $\overrightarrow{P^{*}}\subseteq P^{*}$
to be the subset of all sequences $\mu=(a_{1},e_{1})(a_{2},e_{2})\cdots(a_{n},e_{n})$
where $e_{1}e_{2}\cdots e_{n}$ is a walk in $G$. Given $\mu=(a_{1},e_{1})(a_{2},e_{2})\cdots(a_{n},e_{n})$
in $\overrightarrow{P^{*}}$, we define $\rho(\mu)\coloneqq a_{1}a_{2}\cdots a_{n}$
to be the word obtained by projecting onto $A^{*}$ and let $E(\mu)\coloneqq(i,j)$
where $i$ and $j$ are the initial and terminal vertices, respectively,
of the walk $e_{1}e_{2}\cdots e_{n}$.

For example, consider the monoid network in Figure 1. \begin{center}
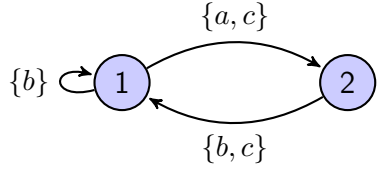

\begin{tikzpicture}[->,>=stealth',shorten >=1pt,auto,node distance=3cm,   thick,main node/.style={circle,fill=blue!20,draw,font=\sffamily}]

\node[main node] (1) {1};
\node[main node] (2) [right of =1] {2};

\path[every node/.style={font=\sffamily\small}]     
(1) edge [bend left] node {$\{a,c\}$} (2)
edge [loop left] node {$\{b\}$} (2)   
(2) edge [bend left] node {$\{b,c\}$} (1);

\end{tikzpicture}
\captionof{figure}{An example of a monoid network}
\end{center}Here $P=\{(b,(1,1)),(a,(1,2)),(c,(1,2)),(b,(2,1)),(c,(2,1))\}$. One
element of $\overrightarrow{P^{*}}$ is $\mu=(b,(2,1))(b,(1,1))(a,(1,2))$,
and so $\rho(\mu)=bba$ and $E(\mu)=(2,2)$.

We say that $(G,P)$ a \textit{monoid network} on $A^{*}$ if for
all nonempty $\mu,\nu\in\overrightarrow{P^{*}}$, if $\rho(\mu)=\rho(\nu)$
and $E(\mu)=E(\nu)$ then $\mu=\nu$. That is, the same word in $A^{*}$
cannot be obtained by traversing two different walks with the same
initial and terminal vertices. It is easy to see that $(G,P)$ in
the example given above is a monoid network.

We can very naturally represent words in $\overrightarrow{P^{*}}$
using matrices. For each element $p=(a,(i,j))\in P$, we associate
$p$ with the $m\times m$ matrix $M_{p}$ with $a$ in the $(i,j)$
entry and 0 everywhere else, which defines a monoid homomorphism $\lambda:P^{*}\rightarrow\Mat_{m}(K\langle\langle A^{*}\rangle\rangle)$,
where we consider the codomain as the multiplicative monoid of the
algebra $\Mat_{m}(K\langle\langle A^{*}\rangle\rangle)$. Applying
$\lambda$ to the empty word $1$ gives the $m\times m$ identity
matrix $I_{m}$.

If $\mu\in\overrightarrow{P^{*}}$ and $E(\mu)=(i,j)$ , then $\lambda(\mu)$
is the $m\times m$ matrix with $\rho(\mu)$ in the $(i,j)$ entry
and 0 everywhere else; we denote this matrix $M_{\mu}$. If $\mu\notin\overrightarrow{P^{*}}$,
then $M_{\mu}=\lambda(\mu)=0_{m}$, the $m\times m$ zero matrix. 

Returning to the example above, the matrices $M_{p}$ are 
\[
\begin{bmatrix}b & 0\\
0 & 0
\end{bmatrix},\begin{bmatrix}0 & a\\
0 & 0
\end{bmatrix},\begin{bmatrix}0 & c\\
0 & 0
\end{bmatrix},\begin{bmatrix}0 & 0\\
b & 0
\end{bmatrix},\,\mathrm{and}\,\begin{bmatrix}0 & 0\\
c & 0
\end{bmatrix},
\]
and for $\mu=(b,(2,1))(b,(1,1))(a,(1,2))$, we have 
\[
\lambda(\mu)=\begin{bmatrix}0 & 0\\
0 & bba
\end{bmatrix}.
\]

We then extend $\lambda$ by linearity to an algebra homomorphism
$K\langle\langle P^{*}\rangle\rangle\rightarrow\Mat_{m}(K\langle\langle A^{*}\rangle\rangle)$,
which we also call $\lambda$ by a slight abuse of notation. Given
a monoid network $(G,P)$ and a subset $S\subseteq A^{*}$, let $\overrightarrow{\Gamma}_{\!\!G}(S)\in\Mat_{m}(K\langle\langle A^{*}\rangle\rangle)$
be the matrix whose $(i,j)$ entry is the generating function for
words in $S$ that can be obtained by traversing a walk from $i$
to $j$ in $G$. It is clear that 
\[
\overrightarrow{\Gamma}_{\!\!G}(S)=\sum_{\mu\in V}M_{\mu}
\]
where $V$ is the set of all words $\mu\in P^{*}$ such that $\rho(\mu)\in S$. 

If the alphabet $A$ is finite, then the idea of monoid networks may
seem too similar to finite-state automata to warrant its own definition,
but our approach is novel and is based on the monoid structure of
$P^{*}$ and the application of the homomorphism $\lambda$. Moreover,
our construction generalizes the combinatorial framework of free monoids,
hence the name ``monoid network''. For example, the following is an
elementary result traditionally proven using the transfer-matrix method
(see \cite[Section 4.7]{Stanley2011} or \cite[Section V.6]{Flajolet2009}),
but we can give a very simple proof using the homomorphism $\lambda$.
\begin{thm}
\label{t-totalgenmat} Suppose that $(G,P)$ is a monoid network on
$A^{*}$. Then 
\[
\overrightarrow{\Gamma}_{\!\!G}(A^{*})=\bigg(I_{m}-\sum_{p\in P}M_{p}\bigg)^{-1}.
\]
\end{thm}
\begin{proof}
Take 
\[
\sum_{\mu\in P^{*}}\mu=\bigg(1-\sum_{p\in P}p\bigg)^{-1},
\]
which is (\ref{e-fmtot}) applied to the free monoid $P^{*}$, and
then apply $\lambda$ to both sides of the equation.
\end{proof}
Our proof of the generalized Goulden\textendash Jackson cluster method
presented later in this section is of a similar flavor. 

Continuing with the example above, we have 
\begin{eqnarray*}
\overrightarrow{\Gamma}_{\!\!G}(A^{*}) & = & \begin{bmatrix}1-b & -a-c\\
-b-c & 1
\end{bmatrix}^{-1}
\end{eqnarray*}
by Theorem \ref{t-totalgenmat}. If we want the generating function
for words by length that can be obtained by traversing a walk from
1 to 2 in $(G,P)$, then we apply to $\overrightarrow{\Gamma}_{\!\!G}(A^{*})$
the homomorphism sending each of the letters to $x$ to obtain the
matrix 
\[
\begin{bmatrix}1-x & -2x\\
-2x & 1
\end{bmatrix}^{-1}=\begin{bmatrix}{\displaystyle \frac{1}{1-x-4x^{2}}\vphantom{{\displaystyle \frac{\frac{dy}{dx}}{\frac{dy}{dx}}}}} & {\displaystyle \frac{2x}{1-x-4x^{2}}}\\
{\displaystyle \frac{2x}{1-x-4x^{2}}}\vphantom{{\displaystyle \frac{\frac{dy}{dx}}{\frac{dy}{dx}}}} & {\displaystyle \frac{1-x}{1-x-4x^{2}}}
\end{bmatrix}
\]
and then take the $(1,2)$ entry.

\subsection{The Goulden\textendash Jackson cluster method for monoid networks}

To motivate our generalization of the Goulden\textendash Jackson cluster
method, let us combine two previous examples and suppose that we want
to count words on the alphabet $A=\{a,b,c\}$ that satisfy two conditions.
First, these words cannot contain any occurrences of $\beta_{1}=acb$
and $\beta_{2}=bc$, and second, these words must be obtainable by
traversing a walk from vertex 1 to vertex 2 in the monoid network
in Figure 1 (from Subsection 3.1).

We can do this using our monoid network version of the Goulden\textendash Jackson
cluster method, which we now present in full generality. Let $A$
be an alphabet and let $B=\{\beta_{1},\beta_{2},\dots\}\subseteq A^{*}$
be a set of words. Moreover, let $(G,P)$ be a monoid network with
$m$ vertices, and for each positive integer $k$, let $\overrightarrow{B}_{\negmedspace k}$
be the set of all words $\mu$ in $\overrightarrow{P^{*}}$ with $\rho(\mu)=\beta_{k}$,
and let $\overrightarrow{B}=\bigcup_{k=1}^{\infty}\overrightarrow{B}_{\negmedspace k}$. 

Define $\overrightarrow{F}_{\!\!G}(t_{1},t_{2},\dots)$ to be the
$m\times m$ matrix whose $(i,j)$ entry is the sum
\[
\sum_{\mu}\rho(\mu)\prod_{k=1}^{\infty}t_{k}^{\occ_{k}(\rho(\mu))}
\]
over all $\mu\in\overrightarrow{P^{*}}$ with $E(\mu)=(i,j)$, which
is the same as the sum 
\[
\sum_{\alpha}\alpha\prod_{k=1}^{\infty}t_{k}^{\occ_{k}(\alpha)}
\]
over all $\alpha\in A^{*}$ that can be obtained by traversing a walk
from vertex $i$ to vertex $j$ in the monoid network $(G,P)$. Furthermore,
define 
\[
\overrightarrow{L}_{\negmedspace G}(t_{1},t_{2},\dots)\coloneqq\sum_{\mu\in\overrightarrow{P^{*}}}M_{\mu}\sum_{c\in C_{\mu}}\prod_{k=1}^{\infty}t_{k}^{\mk_{k}(c)},
\]
where $C_{\mu}$ is the set of all clusters (formed by words in $\overrightarrow{B}$)
on the word $\mu$, and $\mk_{k}(c)$ is the number of marked subwords
in $c$ of the form $(u,\gamma)$ for some position $u$ and some
$\gamma\in\overrightarrow{B}_{\negmedspace k}$. We will refer to
$\overrightarrow{L}_{\negmedspace G}(t_{1},t_{2},\dots)$ as the \textit{cluster
matrix}.
\begin{thm}[Goulden\textendash Jackson cluster method for monoid networks]
\label{t-gjcmmn} Let $A$ be an alphabet and let $B=\{\beta_{1},\beta_{2},\dots\}\subseteq A^{*}$
be a set of words of length at least 2. Also, let $G$ be a digraph
on $[m]$ and let $(G,P)$ be a monoid network on $A^{*}$. Then,
\[
\overrightarrow{F}_{\!\!G}(t_{1},t_{2}\dots)=\bigg(I_{m}-\sum_{p\in P}M_{p}-\overrightarrow{L}_{\negmedspace G}(t_{1}-1,t_{2}-1,\dots)\bigg)^{-1}.
\]
\end{thm}
\begin{proof}
Apply the original Goulden\textendash Jackson cluster method (Theorem
\ref{t-gjcm2}) for the alphabet $P$ and the set $\overrightarrow{B}$,
where we attach the variable $t_{k}$ to each word in $\overrightarrow{B}_{\negmedspace k}$.
Then applying the homomorphism $\lambda$ yields the desired result.
\end{proof}
As before, the set of words in $B$ need not be infinite, and the
number of variables can be less than the number of words in $B$.
It is also possible to alter the cluster matrix to only include clusters
occurring at specified positions in the monoid network, which we do
in Section 4 to count Motzkin paths with no occurrences of subwords
at specified heights. 

We mention three specializations which are completely analogous to
those given after Theorem \ref{t-gjcm1}:
\begin{itemize}
\item By setting each variable equal to 0, we obtain 
\[
\bigg(I_{m}-\sum_{p\in P}M_{p}-\overrightarrow{L}_{\negmedspace G}(-1,-1,\dots)\bigg)^{-1}
\]
as the $m\times m$ matrix whose $(i,j)$ entry is the sum $\sum_{\alpha}\alpha$
over all $\alpha\in A^{*}$ that can be obtained by traversing a walk
from vertex $i$ to vertex $j$ in the monoid network $(G,P)$ and
contain no occurrences of words in $B$.
\item If every word in $B$ has length exactly 2, then setting each variable
equal to 0 yields a monoid network version of the Carlitz\textendash Scoville\textendash Vaughan
theorem.
\item Setting each variable equal to 1 gives an alternative proof for Theorem
\ref{t-totalgenmat}.
\end{itemize}
Observe that the original Goulden\textendash Jackson cluster method
corresponds to the special case in which the monoid network consists
of a single vertex with a loop to which the entire alphabet $A$ is
assigned. Thus Theorem \ref{t-gjcmmn} can accurately be characterized
as a generalization of the Goulden\textendash Jackson cluster method.

Finally, we note that if the alphabet $A$ is finite, then a monoid
network gives the transition diagram of a unambiguous finite automaton.
Unambiguous finite automata are equivalent to deterministic finite
automata, and the transition diagram of a deterministic finite automaton
is a monoid network. Therefore, Theorem \ref{t-gjcmmn} can be used
to count words in a regular language by occurrences of a specified
set of subwords, which has a rational generating function. See \cite{Berstel1988,Flajolet2009,Salomaa1978,Stanley2011}
for several references on the subjects of regular languages, automata,
and rational generating functions.

Let us now complete the example from earlier. We have 
\begin{align*}
\overrightarrow{L}_{\negmedspace G}(t_{1},t_{2}) & =\begin{bmatrix}acb & 0\\
0 & 0
\end{bmatrix}t_{1}+\begin{bmatrix}0 & 0\\
0 & bc
\end{bmatrix}t_{2}+\begin{bmatrix}0 & bc\\
0 & 0
\end{bmatrix}t_{2}+\begin{bmatrix}0 & acbc\\
0 & 0
\end{bmatrix}t_{1}t_{2}\\
 & =\begin{bmatrix}acbt_{1} & bct_{2}+acbct_{1}t_{2}\\
0 & bct_{2}
\end{bmatrix};
\end{align*}
indeed, recall that the only three clusters formed by the words $acb$
and $bc$ are $acb$, $bc$, and $acbc$, which can be obtained in
the given monoid network by traversing walks with initial and terminal
vertices indicated in the matrices above. Thus, 
\begin{align*}
\overrightarrow{F}_{\!\!G}(t_{1},t_{2}) & =\bigg(I_{2}-\sum_{p\in P}M_{p}-\overrightarrow{L}_{\negmedspace G}(t_{1}-1,t_{2}-1)\bigg)^{-1}\\
 & =\left(\begin{bmatrix}1 & 0\\
0 & 1
\end{bmatrix}-\begin{bmatrix}b & a+c\\
b+c & 0
\end{bmatrix}-\begin{bmatrix}acb(t_{1}-1) & bc(t_{2}-1)+acbc(t_{1}-1)(t_{2}-1)\\
0 & bc(t_{2}-1)
\end{bmatrix}\right)^{-1}\\
 & =\begin{bmatrix}1-b-acb(t_{1}-1) & -a-c-bc(t_{2}-1)-acbc(t_{1}-1)(t_{2}-1)\\
-b-c & 1-bc(t_{2}-1)
\end{bmatrix}^{-1}.
\end{align*}
Now we apply the homomorphism sending each of the letters to $x$,
yielding the matrix 
\[
\begin{bmatrix}1-x-x^{3}(t_{1}-1) & -2x-x^{2}(t_{2}-1)-x^{4}(t_{1}-1)(t_{2}-1)\\
-2x & 1-x^{2}(t_{2}-1)
\end{bmatrix}^{-1}
\]
whose $(1,2)$ entry is 
\[
\frac{2x-(1-t_{2})x^{2}+(1-t_{1}-t_{2}+t_{1}t_{2})x^{4}}{1-x-(3+t_{2})x^{2}+(2-t_{1}-t_{2})x^{3}-(1-t_{1}-t_{2}+t_{1}t_{2})x^{5}},
\]
which is the generating function for words obtained by traversing
a walk from vertex 1 to vertex 2 in the given monoid network, weighted
by length, occurrences of $acb$, and occurrences of $bc$. Setting
$t_{1}=t_{2}=0$ gives the generating function 
\[
\frac{2x-x^{2}+x^{4}}{1-x-3x^{2}+2x^{3}-x^{5}}
\]
for those words that do not contain any occurrences of $acb$ or $bc$.

We also state a weighted version of Theorem \ref{t-gjcmmn}. Let $\{\,w_{a}^{(i.j)}\mid(a,(i,j))\in P\,\}$
be a set of weights that commute with each other, the variables $t_{1},t_{2},\dots$,
and the letters in $A$. Set $w_{a}^{(i,j)}=0$ if $(a,(i,j))\notin P$.
Given $\alpha=a_{1}a_{2}\cdots a_{k}\in A^{*}$ and $1\leq i,j\leq m$,
let $w^{(i,j)}(\alpha)=w_{a_{1}}^{e_{1}}\cdots w_{a_{k}}^{e_{k}}$
if there exists $\mu=(a_{1},e_{1})\cdots(a_{k},e_{k})\in\overrightarrow{P^{*}}$
such that $E(\mu)=(i,j)$ and $\rho(\mu)=\alpha$. 

Define the map $\hat{\lambda}:P^{*}\rightarrow\Mat_{m}(K\langle\langle A^{*}\rangle\rangle)$
by sending $p=(a,(i,j))$ to the matrix $\hat{M_{p}}$ with $w_{a}^{(i.j)}a$
in the $(i,j)$ entry and 0 everywhere else. If $\mu=(a_{1},e_{1})\cdots(a_{n},e_{n})\in\overrightarrow{P^{*}}$
and $E(\mu)=(i,j)$, then $\hat{\lambda}(\mu)$\textemdash which we
also denote $\hat{M}_{\mu}$\textemdash has $w_{a_{1}}^{e_{1}}\cdots w_{a_{n}}^{e_{n}}\rho(\mu)$
in the $(i,j)$ entry and 0 everywhere else, and if $\mu\notin\overrightarrow{P^{*}}$
then $\hat{M}_{\mu}=0_{m}$. Again, $\hat{\lambda}$ extends to a
homomorphism $K\langle\langle P^{*}\rangle\rangle\rightarrow\Mat_{m}(K\langle\langle A^{*}\rangle\rangle)$,
which we also call $\hat{\lambda}$. Note that setting all of the
weights equal to 1 gives $\hat{\lambda}=\lambda$.
\begin{thm}[Goulden\textendash Jackson cluster method for monoid networks, weighted
version]
\label{t-gjcmmnw} Let $A$ be an alphabet and let $B=\left\{ \beta_{1},\beta_{2},\dots\right\} \subseteq A^{*}$
be a set of words of length at least 2; let $G$ be a digraph on $[m]$
and let $(G,P)$ be a monoid network on $A^{*}$; let $\hat{F}_{G}(t_{1},t_{2},\dots)$
be the $m\times m$ matrix whose $(i,j)$ entry is the sum 
\[
\sum_{\mu}w^{(i,j)}(\rho(\mu))\rho(\mu)\prod_{k=1}^{\infty}t_{k}^{\occ_{k}(\rho(\mu))}
\]
over all $\mu\in\overrightarrow{P^{*}}$ with $E(\mu)=(i,j)$; and
let 
\[
\hat{L}_{G}(t_{1},\dots,t_{k})\coloneqq\sum_{\mu\in\overrightarrow{P^{*}}}\hat{M}_{\mu}\sum_{c\in C_{\mu}}\prod_{k=1}^{\infty}t_{k}^{\mk_{k}(c)}.
\]
Then, 
\[
\hat{F}_{G}(t_{1},t_{2}\dots)=\bigg(I_{m}-\sum_{p\in P}\hat{M}_{p}-\hat{L}_{G}(t_{1}-1,t_{2}-1,\dots)\bigg)^{-1}.
\]
\end{thm}
\noindent The proof is the same as that of Theorem \ref{t-gjcmmn},
except that we apply $\hat{\lambda}$ instead of $\lambda$.

Although we will not use the weighted version of our main theorem
in subsequent sections, we note that it can be used with the monoid
network framework to examine time-homogeneous Markov chains, which
are probabilistic analogues of finite-state automata. Specifically,
let $(G,P)$ be a monoid network with $m$ vertices, and for every
$a\in A$ and $i,j\in[m]$, let $w_{a}^{(i,j)}\in[0,1]$ such that
$w_{a}^{(i,j)}=0$ if $(a,(i,j))\notin P$ and 
\[
\sum_{j=1}^{m}\sum_{a\in A}w_{a}^{(i,j)}=1
\]
for each fixed $1\leq i\leq m$. With a choice of initial vertex and
terminal vertex, we can think of this monoid network as a random word
model, where a word is given by traversing a random walk in $G$ from
the initial vertex to the terminal vertex with $w_{a}^{(i,j)}$ being
the probability that at vertex $i$, the next letter in the word will
be $a$ and the next arc $(i,j)$. Using Theorem \ref{t-gjcmmnw},
we can then compute probabilities associated with this random process,
such as the probability that a length $n$ word obtained from traversing
a walk between two specified vertices avoids a specified set of forbidden
subwords.

\section{An application to lattice path enumeration}

\subsection{Representing lattice paths using monoid networks}

A \textit{path} on $\mathbb{Z}^{k}$ with \textit{steps} in $S\subseteq\mathbb{Z}^{k}$
is an ordered tuple $(a_{0},a_{1},a_{2},\dots,a_{n})$ of values in
$\mathbb{Z}^{k}$ such that $a_{i+1}-a_{i}\in S$ for every $0\leq i<n$.
Equivalently, it is an ordered tuple $(s_{1},s_{2},\dots,s_{n})$
of values in $S$. Each step $s\in S$ is assigned a length in $\mathbb{Z}$\textemdash which
we take to be 1 unless otherwise noted\textemdash and the \textit{length}
of a path is the sum of the lengths of all of its steps $s_{i}$.

These paths are collectively known as lattice paths. In particular,
lattice paths on $\mathbb{Z}$ have been widely studied in the literature,
usually with the conditions $a_{0}=a_{n}=0$ and $a_{i}\geq0$ for
every $i$. Examples of these paths include \textit{Dyck paths}, which
have steps in $\{-1,1\}$; \textit{Motzkin paths}, which have steps
in $\{-1,0,1\}$; and \textit{Schr\"oder paths}, which are Motzkin
paths but with `0' steps having length 2 instead of 1. These paths
are often illustrated as paths in the plane starting at the origin,
ending on the $x$-axis, and never going below the $x$-axis, with
up steps $(1,1)$ corresponding to $1$, down steps $(1,-1)$ corresponding
to $-1$, and in the case of Motzkin or Schr\"oder paths, flat steps
$(1,0)$ or $(2,0)$, respectively, corresponding to 0. See Figure
2 for an example.\begin{center}
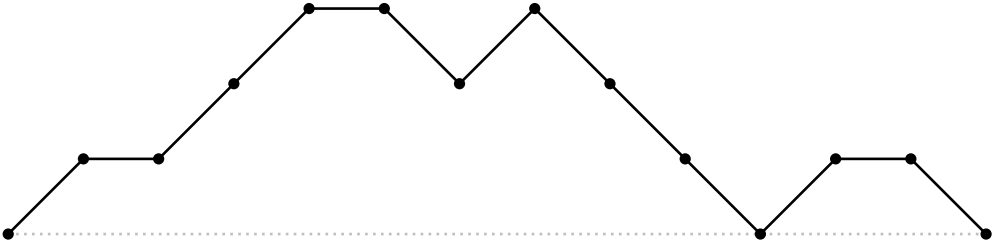

\begin{tikzpicture}[scale=1] 
\draw [line width=0] (4,2); 
\draw[pathcolorlight] (0,0) -- (13,0); 
\drawlinedots{0,1,2,3,4,5,6,7,8,9,10,11,12,13}{0,1,1,2,3,3,2,3,2,1,0,1,1,0} 
\end{tikzpicture}
\captionof{figure}{The Motzkin path corresponding to $UFUUFDUDDDUFD$}
\end{center}

We say that a lattice path on $\mathbb{Z}$ has \textit{height bounded
by} $m$ if we add the condition that $a_{i}\leq m$ for every $i$.
Lattice paths with bounded heights correspond to walks in certain
monoid networks. For example, a Dyck path with height bounded by $m$
is a walk from vertex 0 to itself in the monoid network in Figure
3.

\noindent \begin{center}
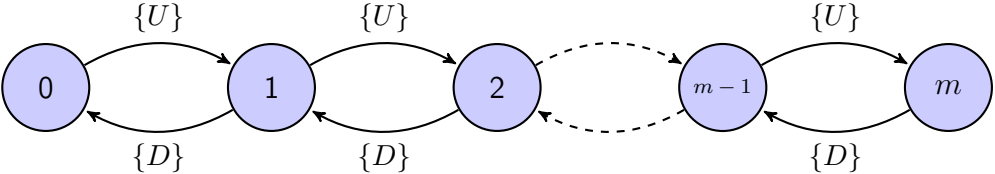

\begin{tikzpicture}[->,>=stealth',shorten >=1pt,auto,node distance=3cm,   thick,main node/.style={circle,fill=blue!20,draw,font=\sffamily, minimum size=2.8em}]

\node[main node] (1) {0};
\node[main node] (2) [right of =1] {1};
\node[main node] (3) [right of =2] {2};
\node[main node] (4) [right of =3] {{\scriptsize{$m-1$}}{\small \par}};
\node[main node] (5) [right of =4] {$m$};

\path[every node/.style={font=\sffamily\small}]     
(1) edge [bend left] node {$\{U\}$} (2)    
(2) edge [bend left] node {$\{U\}$} (3)
edge [bend left] node {$\{D\}$} (1)
(3) edge [bend left] node {$\{D\}$} (2)
edge [dashed, bend left] node {} (4)
(4) edge [bend left] node {$\{U\}$} (5)
edge [dashed, bend left] node {} (3)
(5) edge [bend left] node {$\{D\}$} (4);

\end{tikzpicture}
\captionof{figure}{Dyck path monoid network}
\end{center}Here the alphabet is $\{U,D\}$, with $U$ corresponding to an up
step and $D$ corresponding to a down step. The vertices represent
the possible heights at each step of the path; indeed, a Dyck path
with height bounded by $m$ must begin and end at height 0, and its
height must stay between 0 and $m$.

We can also add a letter $F$ for flat steps, and so we can represent
Motzkin paths and Schr\"oder paths using the monoid network in Figure
4.

\noindent \begin{center}
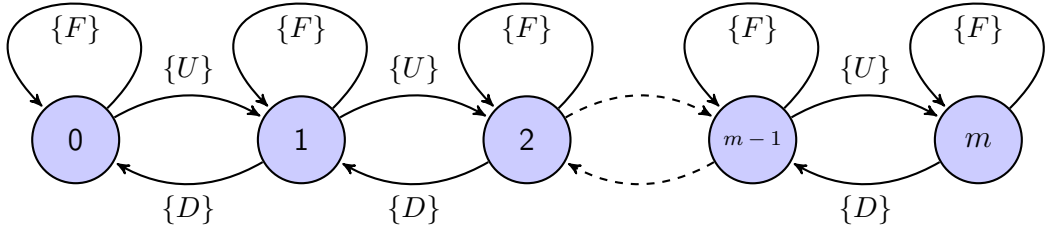

\begin{tikzpicture}[->,>=stealth',shorten >=1pt,auto,node distance=3cm,   thick,main node/.style={circle,fill=blue!20,draw,font=\sffamily, minimum size=2.8em}]

\node[main node] (1) {0};
\node[main node] (2) [right of =1] {1};
\node[main node] (3) [right of =2] {2};
\node[main node] (4) [right of =3] {{\scriptsize{$m-1$}}{\small \par}};
\node[main node] (5) [right of =4] {$m$};

\path[every node/.style={font=\sffamily\small}]     
(1) edge [bend left] node {$\{U\}$} (2)    
edge [loop] node {$\{F\}$} (1)
(2) edge [bend left] node {$\{U\}$} (3)
edge [bend left] node {$\{D\}$} (1)
edge [loop] node {$\{F\}$} (2)
(3) edge [bend left] node {$\{D\}$} (2)
edge [dashed, bend left] node {} (4)
edge [loop] node {$\{F\}$} (3)
(4) edge [bend left] node {$\{U\}$} (5)
edge [dashed, bend left] node {} (3)
edge [loop] node {$\{F\}$} (4)
(5) edge [bend left] node {$\{D\}$} (4)
edge [loop] node {$\{F\}$} (5);

\end{tikzpicture}
\captionof{figure}{Motzkin path monoid network}
\end{center}

Using monoid networks, we can model a wide variety of bounded lattice
paths with different types of steps and various restrictions, so we
may use the tools that we have for monoid networks to obtain generating
functions for counting lattice paths of bounded height. Taking the
formal power series limit as $m\rightarrow\infty$ yields analogous
results for lattice paths of unbounded height.

The idea of representing lattice paths as walks in digraphs and the
transfer-matrix method are standard techniques in lattice path enumeration;
see \cite{Krattenthaler2015} for a recent survey of the literature.
Such an approach has not yet been combined with the Goulden\textendash Jackson
cluster method to count lattice paths by occurrences of subwords,
which we shall do here.

However, the original version of the cluster method was applied by
Wang \cite{Wang2011} to count Dyck paths by occurrences of various
subwords. Using his approach, each cluster formed from the subwords
that one wishes to keep track of is given its own step. For example,
to count Dyck paths by occurrences of $UUD$, the only cluster $UUD$
is replaced by the step $U^{\prime},$ which reduces the problem to
counting paths with steps $U$, $U^{\prime}$, and $D$ that start
at the origin, end on the $x$-axis, and never go below the $x$-axis.
More generally, Wang's method reduces the problem of counting Dyck
paths by occurrences of prescribed subwords to counting paths with
a larger set of steps satisfying the same restrictions, which can
be done by producing recursive decompositions for these paths and
solving the associated functional equations for their generating functions.

Some of our results can be obtained via Wang's method, but there are
two key differences between our method and Wang's. First, our method
allows us to only keep track of subwords that occur only at a prescribed
set of heights, which is not possible using Wang's approach. Moreover,
Wang's approach does not use the correspondence between lattice paths
and walks in digraphs, and also relies on recursive decompositions
of paths which may not always be easy to obtain; our method is more
systematic and reduces almost all of the computations to matrix algebra.
Because Wang conducted his investigation on Dyck paths, we shall instead
focus on Motzkin paths in this paper.

\subsection{A note on continued fractions}

A finite continued fraction is an expression of the form 
\[
a_{0}+\cfrac{b_{1}}{a_{1}+\cfrac{b_{2}}{\ddots+\cfrac{b_{m}}{a_{m}}}},
\]
which we write as 
\[
a_{0}+\frac{b_{1}}{a_{1}+}\:\frac{b_{2}}{a_{2}+}\cdots\frac{b_{m}}{a_{m}}
\]
for compactness. We say that a finite continued fraction has \textit{depth}
$m$ if it is written with $m$ fraction bars when completely written
out in this notation, so the continued fraction above has depth $m$.
We write an infinite continued fraction 
\[
a_{0}+\cfrac{b_{1}}{a_{1}+\cfrac{b_{2}}{a_{2}+\ddots}}
\]
as 
\[
a_{0}+\frac{b_{1}}{a_{1}+}\:\frac{b_{2}}{a_{2}+}\cdots.
\]

Continued fractions arise naturally in combinatorics and especially
in lattice path enumeration; e.g., see Flajolet's landmark paper \cite{Flajolet1980}.
Many of our results in this section are continued fraction formulas.

\subsection{Counting Motzkin paths by ascents}

Let ${\cal M}_{n}^{m}$ be the set of Motzkin paths of length $n$
with height bounded by $m$ and ${\cal M}_{n}$ the set of all Motzkin
paths of length $n$. An \textit{ascent} of a Motzkin path $\mu$
is a maximal consecutive sequence of up steps in $\mu$, and let $\asc(\mu)$
be the number of ascents in $\mu$. We also define 
\[
F_{m}^{\asc}(x,t)\coloneqq\sum_{n=0}^{\infty}\sum_{\mu\in\mathcal{M}_{n}^{m}}t^{\asc(\mu)}x^{n}\qquad\mbox{and}\qquad F^{\asc}(x,t)\coloneqq\sum_{n=0}^{\infty}\sum_{\mu\in\mathcal{M}_{n}}t^{\asc(\mu)}x^{n}
\]
to be bivariate generating functions for Motzkin paths with height
bounded by $m$ and regular Motzkin paths, respectively, weighted
by length and number of ascents. Our main result here is the following
theorem.
\begin{thm}
\label{t-ascents} Let $\{P_{m}^{\asc}(x,t)\}_{m\geq0}$ be the sequence
of polynomials defined by $P_{0}^{\asc}(x,t)=1$, $P_{1}^{\asc}(x,t)=1-x$,
and
\[
P_{m}^{\asc}(x,t)=(1-x-x^{2}(t-1))P_{m-1}^{\asc}(x,t)-(x^{2}+x^{3}(t-1))P_{m-2}^{\asc}(x,t)
\]
for $m\geq2$. Then 
\begin{align*}
F_{m}^{\asc}(x,t) & =\frac{P_{m}^{\asc}(x,t)}{P_{m+1}^{\asc}(x,t)}\\
 & =\underset{\mathrm{depth}\mbox{ }m+1}{\underbrace{\frac{1}{1-x-x^{2}(t-1)-}\:\frac{x^{2}+x^{3}(t-1)}{1-x-x^{2}(t-1)-}\cdots\frac{x^{2}+x^{3}(t-1)}{1-x-x^{2}(t-1)-}\:\frac{x^{2}+x^{3}(t-1)}{1-x}}}
\end{align*}
for $m\geq1$ and 
\begin{align*}
F^{\asc}(x,t) & =\frac{1}{1-x-x^{2}(t-1)-}\:\frac{x^{2}+x^{3}(t-1)}{1-x-x^{2}(t-1)-}\:\frac{x^{2}+x^{3}(t-1)}{1-x-x^{2}(t-1)-}\cdots\\
 & =\frac{1-x-x^{2}(t-1)-\sqrt{1-2x-x^{2}(2t+1)-2x^{3}(t-1)+x^{4}(t-1)^{2}}}{2(x^{2}+x^{3}(t-1))}.
\end{align*}
\end{thm}
\begin{proof}
We apply the cluster method to the Motzkin path monoid network with
$B=\{UD,UF\}$, since the number of occurrences of the subwords $UD$
and $UF$ in a Motzkin path is equal to its number of ascents. We
weight both $UD$ and $UF$ by $t$. The only clusters formed by $UD$
and $UF$ are themselves, and so we have the $(m+1)\times(m+1)$ cluster
matrix 
\[
\overrightarrow{L}_{\negmedspace G}(t)=\begin{bmatrix}UDt & UFt\\
 & UDt & UFt\\
 &  & UDt & \ddots\\
 &  &  & \ddots & \ddots\\
 &  &  &  & UDt & UFt\\
 &  &  &  &  & 0
\end{bmatrix}.
\]
Then, by Theorem \ref{t-gjcmmn}, $\overrightarrow{F}_{\negthinspace\negthinspace G}(t)$
is the inverse matrix of $A_{m}-\overrightarrow{L}_{\negmedspace G}(t-1)$,
where $A_{m}$ is the $(m+1)\times(m+1)$ matrix given by 
\[
A_{m}=\begin{bmatrix}1-F & -U\\
-D & 1-F & -U\\
 & -D & 1-F & \ddots\\
 &  & \ddots & \ddots & \ddots\\
 &  &  & \ddots & 1-F & -U\\
 &  &  &  & -D & 1-F
\end{bmatrix}.
\]
Thus, $F_{m}^{\asc}(x,t)$ is the $(1,1)$ entry of $M_{m}^{-1}$
where $M_{m}$ is the $(m+1)\times(m+1)$ matrix {\footnotesize{}
\[
M_{m}=\begin{bmatrix}1-x-x^{2}(t-1) & -x-x^{2}(t-1)\\
-x & 1-x-x^{2}(t-1) & -x-x^{2}(t-1)\\
 & -x & 1-x-x^{2}(t-1) & \ddots\\
 &  & \ddots & \ddots & \ddots\\
 &  &  & \ddots & 1-x-x^{2}(t-1) & -x-x^{2}(t-1)\\
 &  &  &  & -x & 1-x
\end{bmatrix}
\]
}obtained by applying the homomorphism $U,D,F\mapsto x$ to $A_{m}-\overrightarrow{L}_{\negmedspace G}(t-1)$.{\small{}
}By Cramer's rule, we can compute this generating function as the
quotient of two determinants 
\[
F_{m}^{\asc}(x,t)=\frac{\det M_{m-1}}{\det M_{m}}.
\]
Using column-addition matrix operations, which preserve the determinant,
we can then transform $M_{m}$ into an upper-triangular matrix with
diagonal entries
\[
u_{i,i}=\begin{cases}
1-x-x^{2}(t-1)-\cfrac{x^{2}+x^{3}(t-1)}{u_{i+1,i+1}}, & \mbox{if }1\leq i\leq m\\
1-x, & \mbox{if }i=m+1.
\end{cases}
\]
From here we deduce the recursive expression 
\begin{align*}
\det M_{m}=\prod_{i=1}^{m+1}u_{i,i} & =\left(1-x-x^{2}(t-1)-\frac{x^{2}+x^{3}(t-1)}{\left(\frac{\det M_{m-1}}{\det M_{m-2}}\right)}\right)\det M_{m-1}\\
 & =(1-x-x^{2}(t-1))\det M_{m-1}-(x^{2}+x^{3}(t-1))\det M_{m-2}
\end{align*}
with initial conditions $\det M_{-1}=1$ and $\det M_{0}=1-x$. Hence,
these determinants are polynomials in $x$ and $t$, and we write
$P_{m}^{\asc}(x,t)=\det M_{m-1}$. Moreover, 
\[
\frac{\det M_{m}}{\det M_{m-1}}=\underset{\mathrm{depth}\mbox{ }m}{\underbrace{1-x-x^{2}(t-1)-\frac{x^{2}+x^{3}(t-1)}{1-x-x^{2}(t-1)-}\cdots\frac{x^{2}+x^{3}(t-1)}{1-x-x^{2}(t-1)-}\:\frac{x^{2}+x^{3}(t-1)}{1-x}}},
\]
so 
\begin{equation}
F_{m}^{\asc}(x,t)=\underset{\mathrm{depth}\mbox{ }m+1}{\underbrace{\frac{1}{1-x-x^{2}(t-1)-}\:\frac{x^{2}+x^{3}(t-1)}{1-x-x^{2}(t-1)-}\cdots\frac{x^{2}+x^{3}(t-1)}{1-x-x^{2}(t-1)-}\:\frac{x^{2}+x^{3}(t-1)}{1-x}}}.\label{e-asccf1}
\end{equation}
We now proceed to Motzkin paths unbounded by height. By taking the
limit of (\ref{e-asccf1}) as $m\rightarrow\infty$, this sequence
of formal power series converges to the infinite continued fraction
\begin{align}
F^{\asc}(x,t) & =\frac{1}{1-x-x^{2}(t-1)-}\:\frac{x^{2}+x^{3}(t-1)}{1-x-x^{2}(t-1)-}\:\frac{x^{2}+x^{3}(t-1)}{1-x-x^{2}(t-1)-}\cdots.\label{e-asccf2}
\end{align}
Equation (\ref{e-asccf2}) gives the recursive expression 
\[
F^{\asc}(x,t)=\frac{1}{1-x-x^{2}(t-1)-(x^{2}+x^{3}(t-1))F^{\asc}(x,t)}
\]
or 
\[
(x^{2}+x^{3}(t-1))F^{\asc}(x,t)^{2}-(1-x-x^{2}(t-1))F^{\asc}(x,t)+1=0,
\]
and solving this functional equation gives 
\[
F^{\asc}(x,t)=\frac{1-x-x^{2}(t-1)\pm\sqrt{1-2x-x^{2}(2t+1)-2x^{3}(t-1)+x^{4}(t-1)^{2}}}{2(x^{2}+x^{3}(t-1))}
\]
but one can easily check that the solution given by the minus sign
is the correct one.
\end{proof}
The first several terms of $F^{\asc}(x,t)$ are in the following table:

\begin{center}
\begin{tabular}{c|c|c|c}
$n$ & $[x^{n}]\,F^{\asc}(x,t)$ & $n$ & $[x^{n}]\,F^{\asc}(x,t)$\tabularnewline
\hline 
0 & $1$ & 5 & $1+14t+6t^{2}$\tabularnewline
1 & $1$ & 6 & $1+26t+23t^{2}+t^{3}$\tabularnewline
2 & $1+t$ & 7 & $1+46t+70t^{2}+10t^{3}$\tabularnewline
3 & $1+3t$ & 8 & $1+79t+186t^{2}+56t^{3}+t^{4}$\tabularnewline
4 & $1+7t+t^{2}$ & 9 & $1+133t+451t^{2}+235t^{3}+15t^{4}$\tabularnewline
\end{tabular}
\par\end{center}

\noindent These numbers are in the OEIS \cite[A114580]{oeis}. Notice
that the constant terms of these polynomials are all 1; the only Motzkin
paths with no ascents consist of all flat steps, and there is exactly
one of each length. We also obtain an expression for the linear coefficients,
which count Motzkin paths with exactly one ascent.
\begin{cor}
Let $\Fib(n)$ denote the $n$th Fibonacci number defined by $\Fib(0)=0$,
$\Fib(1)=1$, and $\Fib(n)=\Fib(n-1)+\Fib(n-2)$ for $n\geq2$. Then
the number of Motzkin paths of length $n\geq1$ with exactly one ascent
is equal to $\Fib(n+3)-n-2$.
\end{cor}
\begin{proof}
Using Maple, one may verify that 
\[
\left[\frac{\partial}{\partial t}F^{\asc}(x,t)\right]_{t=0}=\frac{x^{2}}{(1-x-x^{2})(1-x)^{2}}.
\]
It is known that 
\[
\frac{x}{(1-x-x^{2})(1-x)^{2}}
\]
is the generating function for the sequence $\Fib(n+4)-n-3$ (see
\cite[A001924]{oeis}). Then,
\begin{align*}
[x^{n}t]\,F^{\plt}(x,t) & =[x^{n}]\,\left[\frac{\partial}{\partial t}F^{\asc}(x,t)\right]_{t=0}\\
 & =[x^{n-1}]\,\frac{x}{(1-x-x^{2})(1-x)^{2}}\\
 & =\Fib(n+3)-n-2.\tag*{\qedhere}
\end{align*}
\end{proof}
\noindent The leading coefficients of the even-degree polynomials
are 1; a Motzkin path of length $2n$ has at most $n$ ascents, and
only when the path is $(UD)^{n}$. A Motzkin path of length $2n+1$
also has at most $n$ ascents, and we show that the leading coefficients
of the odd-degree polynomials are the triangular numbers.
\begin{prop}
The number of Motzkin paths of length $2n+1$ with $n$ ascents is
${n+2 \choose 2}$.
\end{prop}
\begin{proof}
The maximum number of ascents that a Motzkin path of length $2n+1$
can have is $n$. Fix such a path $\mu$, and let $k$ be the number
of subwords $UD$ that occur at height 0 in $\mu$.

\begin{itemize}
\item If $k=n$, then the remaining step (which must be a flat step) can
be in $k+1$ possible positions: at the beginning, at the end, or
between two consecutive occurrences of $UD$.
\item If $k<n$, then it is easy to see that in order for $\mu$ to have
$n$ ascents, the remaining steps must form the subword $UF(UD)^{n-k-1}D$
beginning at height 0. Again, there are $k+1$ possible positions
for this subword: at the beginning, at the end, or between two consecutive
occurrences of $UD$.
\end{itemize}
Summing over all $k$, we conclude that the number of Motzkin paths
of length $2n+1$ with $n$ ascents is equal to 
\[
\sum_{k=0}^{n}(k+1)={n+2 \choose 2}.\tag*{\qedhere}
\]
\end{proof}
\noindent We can also use the generalized cluster method to count
Motzkin paths with ascents ending only at specified heights. Let $\mathbb{P}$
be the set of positive integers, $\mathbb{N}$ the set of non-negative
integers, $\mathbb{E}$ the set of positive even integers, $\mathbb{O}$
the set of positive odd integers, and $\mathbb{E}_{\geq0}$ the set
of non-negative even integers.
\begin{thm}
\label{t-ascrest} Let $A\subseteq\mathbb{P}$ and let 
\[
F^{\asc}(A;x)\coloneqq\sum_{n=0}^{\infty}c_{n}x^{n}
\]
where $c_{n}$ is the number of Motzkin paths of length $n$ with
every ascent ending at a height in $A$. Then, 
\[
F^{\asc}(A;x)=\frac{1}{1-x+C_{1}-}\:\frac{x^{2}-xC_{1}}{1-x+C_{2}-}\:\frac{x^{2}-xC_{2}}{1-x+C_{3}-}\cdots
\]
where 
\[
C_{i}=\begin{cases}
x^{2}, & \mbox{if }i\notin A\\
0, & \mbox{otherwise}.
\end{cases}
\]
\end{thm}
\begin{proof}
We weight both $UD$ and $UF$ by $t$, but we only wish to consider
instances of these subwords occuring at impermissible heights as we
will be setting $t=0$ afterward. The impermissible heights are $i-1$
where $i\notin A$, so that the corresponding ascents end at height
$i$. Thus, following the proof of Theorem \ref{t-ascents}, we take
the cluster matrix $\overrightarrow{L}_{\negmedspace G}(t)$ but delete
all entries in rows $i-1$ with $i\in A$. We obtain the result by
applying the cluster method, using matrix operations to obtain a continued
fraction formula, and then taking the limit as $m\rightarrow\infty$\textemdash all
in the same way as in the proof of Theorem \ref{t-ascents}\textemdash and
finally by setting $t=0$.
\end{proof}
For example, taking $A=\mathbb{E}$ and $A=\mathbb{O}$, we obtain
\begin{align*}
F^{\asc}(\mathbb{E};x) & =\frac{1}{1-x+x^{2}-}\:\frac{x^{2}-x^{3}}{1-x-}\:\frac{x^{2}}{1-x+x^{2}-}\:\frac{x^{2}-x^{3}}{1-x-}\:\frac{x^{2}}{1-x+x^{2}-}\cdots\\
 & =\frac{1-2x+2x^{2}-\sqrt{1-4x+4x^{2}-4x^{4}+4x^{5}}}{2(x^{2}-x^{3}+x^{4})}\\
 & =1+x+x^{2}+x^{3}+2x^{4}+5x^{5}+12x^{6}+27x^{7}+60x^{8}+135x^{9}+309x^{10}+\cdots
\end{align*}
and 
\begin{align*}
F^{\asc}(\mathbb{O};x) & =\frac{1}{1-x-}\:\frac{x^{2}}{1-x+x^{2}-}\:\frac{x^{2}-x^{3}}{1-x-}\:\frac{x^{2}}{1-x+x^{2}-}\:\frac{x^{2}-x^{3}}{1-x-}\cdots\\
 & =\frac{1-2x+2x^{2}-2x^{3}-\sqrt{1-4x+4x^{2}-4x^{4}+4x^{5}}}{2(x^{2}-2x^{3}+x^{4})}\\
 & =1+x+2x^{2}+4x^{3}+8x^{4}+16x^{5}+33x^{6}+70x^{7}+152x^{8}+336x^{9}+754x^{10}+\cdots
\end{align*}
as the generating functions for Motzkin paths with all ascents ending
at even heights and odd heights, respectively.\footnote{We note that the coefficients of $F^{\asc}(\mathbb{E};x)$ match OEIS
sequence \cite[A190171]{oeis} up to $x^{10}$ and the coefficients
of $F^{\asc}(\mathbb{O};x)$ match OEIS sequence \cite[A110334]{oeis}
up to $x^{12}$, but begin to deviate afterward.}

One can produce a refinement of Theorem \ref{t-ascrest} that also
keeps track of the number of ascents. Rather than deleting rows in
the cluster matrix, assign each $UD$ and $UF$ in those rows a weight
of $u$. After setting $t=0$, the remaining variables $x$ and $u$
would keep track of length and number of ascents, respectively.

It is also possible to count paths with restrictions on the heights
at which ascents begin, but the analysis is slightly more complicated.
Here we would want to set $B=\{DU,FU\}$, which suffices for Motzkin
paths that do not begin with an ascent. However, Motzkin paths that
begin with an ascent can be counted by considering walks in the monoid
network in Figure 5 from vertex $0^{\prime}$ to vertex $0$, and
we would multiply the result by $t$ at the end to take into account
the first ascent.\vspace{-15bp}
 \begin{center}
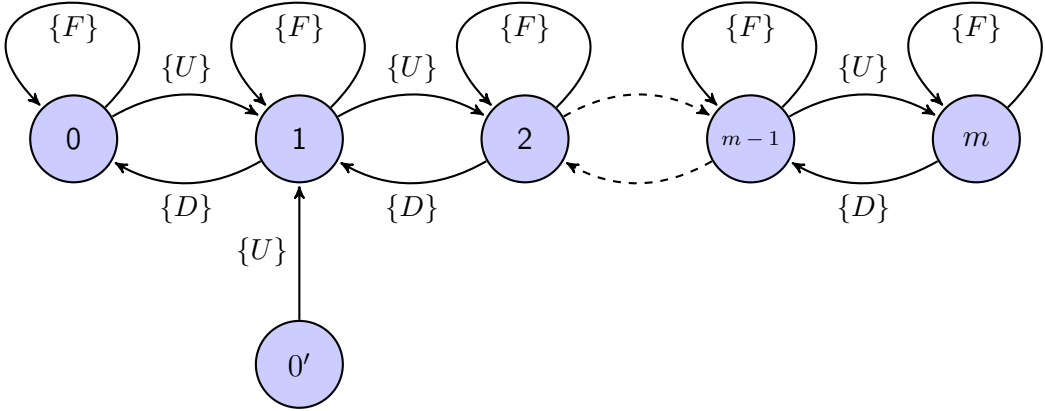

\begin{tikzpicture}[->,>=stealth',shorten >=1pt,auto,node distance=3cm,   thick,main node/.style={circle,fill=blue!20,draw,font=\sffamily, minimum size=2.8em}]

\node[main node] (1) {0};
\node[main node] (2) [right of =1] {1};
\node[main node] (3) [right of =2] {2};
\node[main node] (4) [right of =3] {{\scriptsize{$m-1$}}{\small \par}};
\node[main node] (5) [right of =4] {$m$};
\node[main node] (6) [below of =2] {$0^\prime$};

\path[every node/.style={font=\sffamily\small}]     
(1) edge [bend left] node {$\{U\}$} (2)    
edge [loop] node {$\{F\}$} (1)
(2) edge [bend left] node {$\{U\}$} (3)
edge [bend left] node {$\{D\}$} (1)
edge [loop] node {$\{F\}$} (2)
(3) edge [bend left] node {$\{D\}$} (2)
edge [dashed, bend left] node {} (4)
edge [loop] node {$\{F\}$} (3)
(4) edge [bend left] node {$\{U\}$} (5)
edge [dashed, bend left] node {} (3)
edge [loop] node {$\{F\}$} (4)
(5) edge [bend left] node {$\{D\}$} (4)
edge [loop] node {$\{F\}$} (5)
(6) edge [] node {$\{U\}$} (2);

\end{tikzpicture}
\captionof{figure}{Monoid network for counting Motzkin paths beginning with an ascent}
\end{center}

\subsection{Counting Motzkin paths by plateaus}

We now count Motzkin paths by occurrences of $UF^{k}D$, which we
call a $k$-\textit{plateau}.\footnote{These are sometimes also called $k$-\textit{humps} in the literature.}
For a fixed $k$, let $\plt_{k}(\mu)$ be the number of $k$-plateaus
of a Motzkin path $\mu$, and let 
\[
F_{m}^{\plt_{k}}(x,t)\coloneqq\sum_{n=0}^{\infty}\sum_{\mu\in\mathcal{M}_{n}^{m}}t^{\plt_{k}(\mu)}x^{n}\qquad\mbox{and}\qquad F^{\plt_{k}}(x,t)\coloneqq\sum_{n=0}^{\infty}\sum_{\mu\in\mathcal{M}_{n}}t^{\plt_{k}(\mu)}x^{n}.
\]
Then we have the following formulas.
\begin{thm}
\label{t-kplateau} Let $\{P_{m}^{\plt_{k}}(x,t)\}_{m\geq0}$ be the
sequence of polynomials defined by $P_{0}^{\plt_{k}}(x,t)=1$, $P_{1}^{\plt_{k}}(x,t)=1-x$,
and 
\[
P_{m}^{\plt_{k}}(x,t)=(1-x-x^{k+2}(t-1))P_{m-1}^{\plt_{k}}(x,t)-x^{2}P_{m-2}^{\plt_{k}}(x,t)
\]
for $m\geq2$. Then 
\begin{align*}
F_{m}^{\plt_{k}}(x,t) & =\frac{P_{m}^{\plt_{k}}(x,t)}{P_{m+1}^{\plt_{k}}(x,t)}\\
 & =\underset{\mathrm{depth}\mbox{ }m+1}{\underbrace{\frac{1}{1-x-x^{k+2}(t-1)-}\:\frac{x^{2}}{1-x-x^{k+2}(t-1)-}\cdots\frac{x^{2}}{1-x-x^{k+2}(t-1)-}\:\frac{x^{2}}{1-x}}}
\end{align*}
for $m\geq1$ and 
\begin{align*}
F^{\plt_{k}}(x,t) & =\frac{1}{1-x-x^{k+2}(t-1)-}\:\frac{x^{2}}{1-x-x^{k+2}(t-1)-}\:\frac{x^{2}}{1-x-x^{k+2}(t-1)-}\cdots\\
 & =\frac{1-x-x^{k+2}(t-1)-\sqrt{(1-x-x^{k+2}(t-1))^{2}-4x^{2}}}{2x^{2}}.
\end{align*}
\end{thm}
The two formulas for $F^{\plt_{k}}(x,t)$ were found earlier by Drake
and Gantner \cite[Proposition 3.4 and Theorem 4.2]{Drake2012} using
a different method; here we give a proof using our generalization
of the cluster method.
\begin{proof}
Set $B=\{UF^{k}D\}$, and once again consider the Motzkin path monoid
network. The only cluster formed by $UF^{k}D$ is itself, and so the
$(m+1)\times(m+1)$ cluster matrix is 
\[
\overrightarrow{L}_{\negmedspace G}(t)=\begin{bmatrix}UF^{k}Dt\\
 & UF^{k}Dt\\
 &  & UF^{k}Dt\\
 &  &  & \ddots\\
 &  &  &  & UF^{k}Dt\\
 &  &  &  &  & 0
\end{bmatrix}.
\]
By Theorem \ref{t-gjcmmn}, we have $\overrightarrow{F}_{\negthinspace\negthinspace G}(t)=(A_{m}-\overrightarrow{L}_{\negmedspace G}(t))^{-1}$
(where $A_{m}$ is defined in the proof of Theorem \ref{t-ascents}),
and so $F_{m}^{\plt_{k}}(x,t)$ is the $(1,1)$ entry of $M_{m}^{-1}$
where $M_{m}$ is the matrix {\footnotesize{}
\[
M_{m}=\begin{bmatrix}1-x-x^{k+2}(t-1) & -x\\
-x & 1-x-x^{k+2}(t-1) & -x\\
 & -x & 1-x-x^{k+2}(t-1) & \ddots\\
 &  & \ddots & \ddots & \ddots\\
 &  &  & \ddots & 1-x-x^{k+2}(t-1) & -x\\
 &  &  &  & -x & 1-x
\end{bmatrix}
\]
}obtained by applying to $A_{m}-\overrightarrow{L}_{\negmedspace G}(t)$
the homomorphism sending each of $U$, $F$, and $D$ to $x$.

It follows that

\[
F_{m}^{\plt_{k}}(x,t)=\frac{\det M_{m-1}}{\det M_{m}},
\]
and the determinant of $M_{m}$ is equal to that of an upper-triangular
matrix with diagonal entries
\[
u_{i,i}=\begin{cases}
1-x-x^{k+2}(t-1)-{\displaystyle \frac{x^{2}}{u_{i+1,i+1}}}, & \mbox{if }1\leq i\leq m\\
1-x. & \mbox{if }i=m+1.
\end{cases}
\]
Thus we have the recursion {\allowdisplaybreaks} 
\begin{align*}
\det M_{m}=\prod_{i=1}^{m+1}u_{i,i} & =\left(1-x-x^{k+2}(t-1)-\frac{x^{2}}{\left(\frac{\det M_{m-1}}{\det M_{m-2}}\right)}\right)\det M_{m-1}\\
 & =(1-x-x^{k+2}(t-1))\det M_{m-1}-x^{2}\det M_{m-2}
\end{align*}
with initial conditions $\det M_{-1}=1$ and $\det M_{0}=1-x$. These
are polynomials in $x$ and $t$, and we write $P_{m}^{\plt_{k}}(x,t)=\det M_{m-1}$.
Moreover, 
\[
\frac{\det M_{m}}{\det M_{m-1}}=\underset{\mathrm{depth}\mbox{ }m}{\underbrace{1-x-x^{k+2}(t-1)-\frac{x^{2}}{1-x-x^{k+2}(t-1)-}\cdots\frac{x^{2}}{1-x-x^{k+2}(t-1)-}\:\frac{x^{2}}{1-x}}},
\]
so 
\[
F_{m}^{\plt_{k}}(x,t)=\underset{\mathrm{depth}\mbox{ }m+1}{\underbrace{\frac{1}{1-x-x^{k+2}(t-1)-}\:\frac{x^{2}}{1-x-x^{k+2}(t-1)-}\cdots\frac{x^{2}}{1-x-x^{k+2}(t-1)-}\:\frac{x^{2}}{1-x}}}.
\]
Taking the limit as $m\rightarrow\infty$, we obtain 
\begin{align*}
F^{\plt_{k}}(x,t) & =\frac{1}{1-x-x^{k+2}(t-1)-}\:\frac{x^{2}}{1-x-x^{k+2}(t-1)-}\:\frac{x^{2}}{1-x-x^{k+2}(t-1)-}\cdots\\
 & =\frac{1}{1-x-x^{k+2}(t-1)-x^{2}F^{\plt_{k}}(x,t)}
\end{align*}
which can be rewritten as 
\begin{equation}
x^{2}F^{\plt_{k}}(x,t)^{2}-(1-x-x^{k+2}(t-1))F^{\plt_{k}}(x,t)+1=0.\label{e-pltrec}
\end{equation}
Solving (\ref{e-pltrec}) gives 
\[
F^{\plt_{k}}(x,t)=\frac{1-x-x^{k+2}(t-1)\pm\sqrt{(1-x-x^{k+2}(t-1))^{2}-4x^{2}}}{2x^{2}},
\]
but one can check that the solution given by the minus sign is the
correct one.
\end{proof}
By specializing to $k=0$ and defining $\pk=\plt_{0}$, we obtain
the bivariate generating function 
\[
F^{\pk}(x,t)=\frac{1-x-x^{2}(t-1)-\sqrt{(1-x-x^{2}(t-1))^{2}-4x^{2}}}{2x^{2}}
\]
counting Motzkin paths by \textit{peaks}, which are occurrences of
$UD$. The first several terms of $F^{\pk}(x,t)$ are in the following
table:

\begin{center}
\begin{tabular}{c|c|c|c}
$n$ & $[x^{n}]\,F^{\pk}(x,t)$ & $n$ & $[x^{n}]\,F^{\pk}(x,t)$\tabularnewline
\hline 
0 & $1$ & 5 & $8+10t+3t^{2}$\tabularnewline
1 & $1$ & 6 & $17+24t+9t^{2}+t^{3}$\tabularnewline
2 & $1+t$ & 7 & $37+58t+28t^{2}+4t^{3}$\tabularnewline
3 & $2+2t$ & 8 & $82+143t+81t^{2}+16t^{3}+t^{4}$\tabularnewline
4 & $4+4+t^{2}$ & 9 & $185+354t+231t^{2}+60t^{3}+5t^{4}$\tabularnewline
\end{tabular}
\par\end{center}

\noindent See \cite[A097860]{oeis} for its OEIS entry. Also see \cite[A004148]{oeis}
for the constant coefficients of these polynomials, which count Motzkin
paths with no peaks. The generating function for the linear coefficients
of these polynomials can be verified to be 
\[
\left[\frac{\partial}{\partial t}F^{\pk}(x,t)\right]_{t=0}=\frac{1-x+x^{2}-\sqrt{1-2x-x^{2}-2x^{3}+x^{4}}}{2\sqrt{1-2x-x^{2}-2x^{3}+x^{4}}},
\]
and interestingly enough, dividing this generating function by $x$
(i.e., shifting the indices of the underlying sequence) yields the
generating function for the number of flat steps in all peakless Motzkin
paths of length $n$ (see \cite[A110236]{oeis}). These numbers are
given by a binomial coefficient sum, which in turn gives us the following
corollary.
\begin{cor}
The number of Motzkin paths of length $n\geq2$ with exactly one peak
is equal to $\sum_{k=0}^{n-2}{k+1 \choose n-k+1}{k \choose n-k}$.
\end{cor}
Now let us consider 1-plateaus, or occurrences of $UFD$. The bivariate
generating function 
\[
F^{\plt_{1}}(x,t)=\frac{1-x-x^{3}(t-1)-\sqrt{(1-x-x^{3}(t-1))^{2}-4x^{2}}}{2x^{2}}
\]
counts Motzkin paths by 1-plateaus, and its first several terms are:

\begin{center}
\begin{tabular}{c|c|c|c}
$n$ & $[x^{n}]\,F^{\plt_{1}}(x,t)$ & $n$ & $[x^{n}]\,F^{\plt_{1}}(x,t)$\tabularnewline
\hline 
0 & $1$ & 5 & $15+6t$\tabularnewline
1 & $1$ & 6 & $36+14t+t^{2}$\tabularnewline
2 & $2$ & 7 & $85+39t+3t^{2}$\tabularnewline
3 & $3+t$ & 8 & $209+102t+12t^{2}$\tabularnewline
4 & $7+2t$ & 9 & $517+280t+37t^{2}+t^{3}$\tabularnewline
\end{tabular}
\par\end{center}

\noindent These are also in the OEIS \cite[A114583]{oeis}, along
with the constant coefficients \cite[A114584]{oeis}, which count
Motzkin paths with no occurrences of $UFD$.

We can also count Motzkin paths by all plateaus, without a fixed $k$.
Let $\plt(\mu)$ be the number of plateaus in a Motzkin path $\mu$,
that is, the number of occurrences of subwords in $B=\{UD,UFD,UFFD,\dots\}$.
We define the bivariate generating functions $F_{m}^{\plt}(x,t)$
and $F^{\plt}(x,t)$ in the analogous way as before, and to find expressions
for these generating functions, we would change each nonzero entry
in the cluster matrix from $UF^{k}Dt$ (for a fixed $k$) to 
\[
\sum_{k=0}^{\infty}UF^{k}Dt=U(1-F)^{-1}Dt.
\]
Then the computation would follow in the same way, yielding the following
result.
\begin{thm}
\label{t-plateau} Let $\{R_{m}^{\plt}(x,t)\}_{m\geq0}$ be the sequence
of rational functions defined by $R_{0}^{\plt}(x,t)=1$, $R_{1}^{\plt}(x,t)=1-x$,
and
\[
R_{m}^{\plt}(x,t)=\left(1-x-\frac{x^{2}}{1-x}(t-1)\right)R_{m-1}^{\plt}(x,t)-x^{2}R_{m-2}^{\plt}(x,t)
\]
for $m\geq2$. Then 
\begin{align*}
F_{m}^{\plt}(x,t) & =\frac{R_{m}^{\plt}(x,t)}{R_{m+1}^{\plt}(x,t)}\\
 & =\underset{\mathrm{depth}\mbox{ }m+1}{\underbrace{\frac{1}{1-x-\frac{x^{2}}{1-x}(t-1)-}\:\frac{x^{2}}{1-x-\frac{x^{2}}{1-x}(t-1)-}\cdots\frac{x^{2}}{1-x-\frac{x^{2}}{1-x}(t-1)-}\:\frac{x^{2}}{1-x}}}
\end{align*}
for $m\geq1$ and 
\begin{align*}
F^{\plt}(x,t) & =\frac{1}{1-x-\frac{x^{2}}{1-x}(t-1)-}\:\frac{x^{2}}{1-x-\frac{x^{2}}{1-x}(t-1)-}\:\frac{x^{2}}{1-x-\frac{x^{2}}{1-x}(t-1)-}\cdots\\
 & =\frac{1-2x-x^{2}(t-2)-\sqrt{1-4x-2x^{2}(t-2)+4x^{3}t+x^{4}t(t-4)}}{2(x^{2}-x^{3})}.
\end{align*}
\end{thm}
The first several terms of $F^{\plt}(x,t)$ are below, which can also
be found on the OEIS \cite[A097229]{oeis}:

\begin{center}
\begin{tabular}{c|c|c|c}
$n$ & $[x^{n}]\,F^{\plt}(x,t)$ & $n$ & $[x^{n}]\,F^{\plt}(x,t)$\tabularnewline
\hline 
0 & $1$ & 5 & $1+15t+5t^{2}$\tabularnewline
1 & $1$ & 6 & $1+31t+18t^{2}+t^{3}$\tabularnewline
2 & $1+t$ & 7 & $1+63t+56t^{2}+7t^{3}$\tabularnewline
3 & $1+3t$ & 8 & $1+127t+160t^{2}+34t^{3}+t^{4}$\tabularnewline
4 & $1+7t+t^{2}$ & 9 & $1+255t+432t^{2}+138t^{3}+9t^{4}$\tabularnewline
\end{tabular}
\par\end{center}

\noindent We now give expressions for the linear and quadratic coefficients
of these polynomials.
\begin{cor}
\noindent The number of Motzkin paths of length $n\geq1$ with exactly
one plateau is equal to $2^{n-1}-1$.
\end{cor}
\begin{proof}
Using Maple, one may verify that 
\[
\left[\frac{\partial}{\partial t}F^{\plt}(x,t)\right]_{t=0}=\frac{x^{2}}{(1-2x)(1-x)}.
\]
{\allowdisplaybreaks} Then, 
\begin{align*}
[x^{n}t]\,F^{\plt}(x,t) & =[x^{n}]\,\left[\frac{\partial}{\partial t}F^{\plt}(x,t)\right]_{t=0}\\
 & =[x^{n-2}]\,\frac{1}{(1-2x)(1-x)}\\
 & =[x^{n-2}]\,\left(\frac{2}{1-2x}-\frac{1}{1-x}\right)\\
 & =[x^{n-2}]\Big(\sum_{n=0}^{\infty}(2^{n+1}-1)x^{n}\Big)\\
 & =2^{n-1}-1.\tag*{\qedhere}
\end{align*}
\end{proof}
\begin{cor}
\noindent The number of Motzkin paths of length $n\geq3$ with exactly
two plateaus is equal to $(n-3)n2^{n-6}$.
\end{cor}
\begin{proof}
Using Maple, one may verify that 
\[
\left[\frac{\partial^{2}}{\partial t^{2}}F^{\plt}(x,t)\right]_{t=0}=\frac{2(1-x)x^{4}}{(1-2x)^{3}}
\]
and $(1-x)x/(1-2x)^{3}$ is known to be the generating function for
the sequence $(n(n+3)2^{n-3})_{n\geq1}$ (see \cite[A001793]{oeis}).
Then, 
\begin{align*}
[x^{n}t^{2}]\,F^{\plt}(x,t) & =[x^{n}]\,\frac{1}{2}\left[\frac{\partial^{2}}{\partial t^{2}}F^{\plt}(x,t)\right]_{t=0}\\
 & =[x^{n-3}]\,\frac{(1-x)x}{(1-2x)^{3}}\\
 & =(n-3)n2^{n-6}.\tag*{\qedhere}
\end{align*}
\end{proof}
Hence, Motzkin paths with exactly 1 plateau and those with exactly
2 plateaus are equinumerous with many other combinatorial objects
(see \cite[A000225 and A001793]{oeis}).

Drake and Gantner \cite[Section 4]{Drake2012} showed how one can
find continued fraction formulas for variations of these results,
including bivariate generating functions for counting Motzkin paths
by plateaus occurring only at certain heights, and with restrictions
on the lengths of plateaus. Their approach involved inserting appropriate
``correction terms'' at each level of the continued fraction formulas
that encode the types of plateaus that they wish to count.

All of these variations can also be computed using our method. To
disregard plateaus occurring at certain heights, we would delete the
corresponding rows from the cluster matrix, which is completely analogous
to Theorem \ref{t-ascrest} for ascents. To place restrictions on
the lengths of plateaus, we would alter the ``forbidden set'' $B$
appropriately and set the appropriate variables to 0. We leave the
details to the reader.

Our method also allows for an interpretation of Drake and Gantner's
correction terms in terms of clusters. Their correction terms are
of the form $x^{k}(t-1)$ for various $k$ and are then multiplied
by $x^{2}$, and these precisely correspond to the terms contributed
by the cluster matrix in our computations. This is a relatively simple
case because the only clusters formed by the words in $B=\{UD,UFD,UFFD,\dots\}$
are the words in $B$ themselves. Counting paths by subwords having
additional clusters would require more complicated correction terms
when working through the lens of Drake and Gantner.

\subsection{Counting Motzkin paths by peaks and valleys}

Peaks, or occurrences of $UD$, were introduced in the previous subsection.
Similarly, we define a \textit{valley} to be an occurrence of $DU$,
and $\val(\mu)$ the number of valleys of a Motzkin path $\mu$. Here
we find the joint distribution of peaks and valleys in Motzkin paths.
Let 
\[
F_{m}^{\mathrm{p,v}}(x,t_{1},t_{2})\coloneqq\sum_{n=0}^{\infty}\sum_{\mu\in\mathcal{M}_{n}^{m}}t_{1}^{\pk(\mu)}t_{2}^{\val(\mu)}x^{n}\qquad\mbox{and}\qquad F^{\mathrm{p,v}}(x,t_{1},t_{2})\coloneqq\sum_{n=0}^{\infty}\sum_{\mu\in\mathcal{M}_{n}}t_{1}^{\pk(\mu)}t_{2}^{\val(\mu)}x^{n}.
\]
Then we have the following theorem.
\begin{thm}
\label{t-pkval} Let $\{R_{m}^{\mathrm{p,v}}(x,t_{1},t_{2})\}_{m\geq0}$
be the sequence of rational functions defined by $R_{0}^{\mathrm{p,v}}(x,t_{1},t_{2})=1$,
$R_{1}^{\mathrm{p,v}}(x,t_{1},t_{2})=1-x-C_{2}$, and 
\[
R_{m}^{\mathrm{p,v}}(x,t_{1},t_{2})=(1-x-C_{1}-C_{2})R_{m-1}^{\mathrm{p,v}}(x,t_{1},t_{2})-(x+C_{3})^{2}R_{m-2}^{\mathrm{p,v}}(x,t_{1},t_{2})
\]
for $m\geq2$, where
\[
C_{1}=\frac{x^{2}(t_{1}-1)}{1-x^{2}(t_{1}-1)(t_{2}-1)},\;C_{2}=\frac{x^{2}(t_{2}-1)}{1-x^{2}(t_{1}-1)(t_{2}-1)},\;\mathrm{and}\;C_{3}=\frac{x^{3}(t_{1}-1)(t_{2}-1)}{1-x^{2}(t_{1}-1)(t_{2}-1)}.
\]
Then 
\begin{align*}
F_{m}^{\mathrm{p,v}}(x,t_{1},t_{2}) & =\frac{R_{m}^{\mathrm{p,v}}(x,t_{1},t_{2})}{(1-x-C_{1})R_{m}^{\mathrm{p,v}}(x,t_{1},t_{2})-(x+C_{3})^{2}R_{m-1}^{\mathrm{p,v}}(x,t_{1},t_{2})}\\
 & =\underset{\mathrm{depth}\mbox{ }m+1}{\underbrace{\frac{1}{1-x-C_{1}-}\:\frac{(x+C_{3})^{2}}{1-x-C_{1}-C_{2}-}\cdots\frac{(x+C_{3})^{2}}{1-x-C_{1}-C_{2}-}\:\frac{(x+C_{3})^{2}}{1-x-C_{2}}}}
\end{align*}
for $m\geq1$ and 
\begin{align*}
F^{\mathrm{p,v}}(x,t_{1},t_{2}) & =\frac{1}{1-x-C_{1}-}\:\frac{(x+C_{3})^{2}}{1-x-C_{1}-C_{2}-}\:\frac{(x+C_{3})^{2}}{1-x-C_{1}-C_{2}-}\cdots\\
 & =\frac{2}{1-x-C_{1}+C_{2}+\sqrt{(1-x-C_{1}-C_{2})^{2}-4(x+C_{3})^{2}}}.
\end{align*}
\end{thm}
\begin{proof}
Set $B=\{UD,DU\}$. This time, we weight occurrences of $UD$ by $t_{1}$
and occurrences of $DU$ by $t_{2}$. However, finding the cluster
matrix is no longer a trivial task. We make the following observations:

\begin{itemize}
\item Clusters starting and ending at height 0 are of the form $UDUD\cdots UD$,
since a path cannot go down from height 0. We can decompose these
words into a sequence of $UD$s, where the first $UD$ contributes
a $t_{1}$ and each subsequent $UD$ contributes a $t_{1}$ and a
$t_{2}$.
\item Clusters starting and ending at height $m$ are of the form $DUDU\cdots DU$,
since a path cannot go up from height $m$. We can decompose these
words into a sequence of $DU$s, where the first $DU$ contributes
a $t_{2}$ and each subsequent $DU$ contributes a $t_{1}$ and a
$t_{2}$.
\item Clusters starting and ending at height $k$ with $0<k<m$ are of the
above two forms, since a path can go either up or down from height
$k$.
\item Clusters starting at height $k$ and ending at height $k+1$ are of
the form $UDUDU\cdots DU$, which can be decomposed into an initial
subword $UDU$\textemdash contributing a $t_{1}$ and a $t_{2}$\textemdash and
a sequence of $DU$s, each contributing a $t_{1}$ and a $t_{2}$.
\item Clusters starting at height $k$ and ending at height $k-1$ are of
the form $DUDUD\cdots UD$, which can be decomposed into an initial
subword $DUD$\textemdash contributing a $t_{1}$ and a $t_{2}$\textemdash and
a sequence of $UD$s, each contributing a $t_{1}$ and a $t_{2}$.
\end{itemize}
Thus, the $(m+1)\times(m+1)$ cluster matrix is 
\[
\overrightarrow{L}_{\negmedspace G}(t_{1},t_{2})=\begin{bmatrix}\hat{C_{1}} & \hat{C_{3}}\\
\hat{C_{4}} & \hat{C_{1}}+\hat{C_{2}} & \hat{C_{3}}\\
 & \hat{C_{4}} & \hat{C_{1}}+\hat{C_{2}} & \ddots\\
 &  & \ddots & \ddots & \ddots\\
 &  &  & \ddots & \hat{C_{1}}+\hat{C_{2}} & \hat{C_{3}}\\
 &  &  &  & \hat{C_{4}} & \hat{C_{2}}
\end{bmatrix}
\]
where 
\[
\hat{C_{1}}=\frac{UDt_{1}}{1-UDt_{1}t_{2}},\;\hat{C_{2}}=\frac{DUt_{2}}{1-DUt_{1}t_{2}},\;\hat{C_{3}}=\frac{UDUt_{1}t_{2}}{1-DUt_{1}t_{2}},\;\mathrm{and}\;\hat{C_{4}}=\frac{DUDt_{1}t_{2}}{1-UDt_{1}t_{2}}.
\]

By applying Theorem \ref{t-gjcmmn}, we see that $F_{m}^{\mathrm{p,v}}(x,t_{1},t_{2})$
is the $(1,1)$ entry of $M_{m}^{-1}$ where $M_{m}$ is the $(m+1)\times(m+1)$
matrix {\small{}
\[
M_{m}=\begin{bmatrix}1-x-C_{1} & -x-C_{3}\\
-x-C_{3} & 1-x-C_{1}-C_{2} & -x-C_{3}\\
 & -x-C_{3} & 1-x-C_{1}-C_{2} & \ddots\\
 &  & \ddots & \ddots & \ddots\\
 &  &  & \ddots & 1-x-C_{1}-C_{2} & -x-C_{3}\\
 &  &  &  & -x-C_{3} & 1-x-C_{2}
\end{bmatrix}
\]
}and $C_{1}$, $C_{2}$, and $C_{3}$ are defined in the statement
of this theorem. Then, 
\begin{align*}
F_{m}^{\mathrm{p,v}}(x,t_{1},t_{2})=\frac{\det M_{m}^{\prime}}{\det M_{m}} & =\frac{\det M_{m}^{\prime}}{\left(1-x-C_{1}-\frac{(x+C_{3})^{2}}{\left(\frac{\det M_{m}^{\prime}}{\det M_{m-1}^{\prime}}\right)}\right)\det M_{m}^{\prime}}\\
 & =\frac{\det M_{m}^{\prime}}{(1-x-C_{1})\det M_{m}^{\prime}-(x+C_{3})^{2}\det M_{m-1}^{\prime}}
\end{align*}
where $M_{m}^{\prime}$ is the matrix obtained from $M_{m}$ by deleting
the first row and the first column. The determinant of $M_{m}^{\prime}$
is equal to that of an upper-triangular matrix with diagonal entries
\[
u_{i,i}=\begin{cases}
1-x-C_{1}-C_{2}-\cfrac{(x+C_{3})^{2}}{u_{i+1,i+1}}, & \mbox{if }1\leq i\leq m\\
1-x-C_{2}, & \mbox{if }i=m+1,
\end{cases}
\]
so {\allowdisplaybreaks}
\begin{align*}
\det M_{m}^{\prime}=\prod_{i=1}^{m+1}u_{i,i} & =\left(1-x-C_{1}-C_{2}-\frac{(x+C_{3})^{2}}{\left(\frac{\det M_{m-1}^{\prime}}{\det M_{m-2}^{\prime}}\right)}\right)\det M_{m-1}^{\prime}\\
 & =(1-x-C_{1}-C_{2})\det M_{m-1}^{\prime}-(x+C_{3})^{2}\det M_{m-2}^{\prime}
\end{align*}
with initial conditions $\det M_{0}^{\prime}=1$ and $\det M_{1}^{\prime}=1-x-C_{2}$.
These are rational functions in $x$, $t_{1}$, and $t_{2}$; we write
$R_{m}^{\mathrm{p,v}}(x,t_{1},t_{2})=\det M_{m}^{\prime}$. Furthermore,

\[
\frac{\det M_{m}}{\det M_{m}^{\prime}}=\underset{\mathrm{depth}\mbox{ }m}{\underbrace{1-x-C_{1}-\frac{(x+C_{3})^{2}}{1-x-C_{1}-C_{2}-}\cdots\frac{(x+C_{3})^{2}}{1-x-C_{1}-C_{2}-}\:\frac{(x+C_{3})^{2}}{1-x-C_{2}}}},
\]
so 
\[
F_{m}^{\mathrm{p,v}}(x,t_{1},t_{2})=\underset{\mathrm{depth}\mbox{ }m+1}{\underbrace{\frac{1}{1-x-C_{1}-}\:\frac{(x+C_{3})^{2}}{1-x-C_{1}-C_{2}-}\cdots\frac{(x+C_{3})^{2}}{1-x-C_{1}-C_{2}-}\:\frac{(x+C_{3})^{2}}{1-x-C_{2}}}}.
\]
By taking the limit as $m\rightarrow\infty$, we have that 
\begin{align*}
F^{\mathrm{p,v}}(x,t_{1},t_{2}) & =\frac{1}{1-x-C_{1}-}\:\frac{(x+C_{3})^{2}}{1-x-C_{1}-C_{2}-}\:\frac{(x+C_{3})^{2}}{1-x-C_{1}-C_{2}-}\cdots\\
 & =\frac{1}{1-x-C_{1}-(x+C_{3})^{2}G(x,t_{1},t_{2})}
\end{align*}
where 
\begin{align*}
G(x,t_{1},t_{2}) & =\frac{1}{1-x-C_{1}-C_{2}}\:\frac{(x+C_{3})^{2}}{1-x-C_{1}-C_{2}-}\:\frac{(x+C_{3})^{2}}{1-x-C_{1}-C_{2}-}\cdots\\
 & =\frac{1}{1-x-C_{1}-C_{2}-(x+C_{3})^{2}G(x,t_{1},t_{2})}.
\end{align*}
Thus we have the functional equation 
\[
(x+C_{3})^{2}G(x,t_{1},t_{2})^{2}-(1-x-C_{1}-C_{2})G(x,t_{1},t_{2})+1=0,
\]
and solving it gives 
\[
G(x,t_{1},t_{2})=\frac{1-x-C_{1}-C_{2}\pm\sqrt{(1-x-C_{1}-C_{2})^{2}-4(x+C_{3})^{2}}}{2(x+C_{3})^{2}}.
\]
As before, one can verify that the solution given by the minus sign
is the correct one, and we conclude that 
\begin{align*}
F^{\mathrm{p,v}}(x,t_{1},t_{2}) & =\frac{1}{1-x-C_{1}-\frac{1}{2}\Big(1-x-C_{1}-C_{2}-\sqrt{(1-x-C_{1}-C_{2})^{2}-4(x+C_{3})^{2}}\Big)}\\
 & =\frac{2}{1-x-C_{1}+C_{2}+\sqrt{(1-x-C_{1}-C_{2})^{2}-4(x+C_{3})^{2}}}.\tag*{\qedhere}
\end{align*}
\end{proof}
The first several terms of $F^{\mathrm{p,v}}(x,t_{1},t_{2})$ are
the following:

\begin{center}
\begin{tabular}{c|c}
$n$ & $[x^{n}]\,F^{\mathrm{p,v}}(x,t_{1},t_{2})$\tabularnewline
\hline 
0 & $1$\tabularnewline
1 & $1$\tabularnewline
2 & $1+t_{1}$\tabularnewline
3 & $2+2t_{1}$\tabularnewline
4 & $4+4t_{1}+t_{1}^{2}t_{2}$\tabularnewline
5 & $8+8t_{1}+2t_{1}t_{2}+t_{1}^{2}+2t^{2}t_{2}$\tabularnewline
6 & $16+t_{2}+18t_{1}+6t_{1}t_{2}+3t_{1}^{2}+6t_{1}^{2}t_{2}+t_{1}^{3}t_{2}^{2}$\tabularnewline
7 & $33+4t_{2}+40t_{1}+18t_{1}t_{2}+9t_{1}^{2}+16t_{1}^{2}t_{2}+3t_{1}^{2}t_{2}^{2}+2t_{1}^{3}t_{2}+2t_{1}^{3}t_{2}^{2}$\tabularnewline
8 & $69+13t_{2}+90t_{1}+50t_{1}t_{2}+25t_{1}^{2}+3t_{1}t_{2}^{2}+47t_{1}^{2}t_{2}+t_{1}^{3}+9t_{1}^{2}t_{2}^{2}+6t_{1}^{3}t_{2}+9t_{1}^{3}t_{2}^{2}+t_{1}^{4}t_{2}^{3}$\tabularnewline
\end{tabular}
\par\end{center}

\noindent The constant coefficients, which count Motzkin paths with
no peaks and valleys, are in the OEIS \cite[A004149]{oeis}.

Liu, Ma, and Yeh \cite{Liu2008} gave recursive and continued fraction
formulas for counting Dyck paths with peaks avoiding a specified set
of heights and valleys avoiding another specified set of heights.
We can do the same thing by applying our cluster method to the monoid
network for Dyck paths, but here we give the analogous results for
Motzkin paths.\footnote{Liu, Ma, and Yeh defined the height of a peak (respectively, valley)
to be the height at which its down step (respectively, up step) occurs,
but we use the convention that the height of a peak or valley is the
height at which the corresponding subword ($UD$ or $DU$) begins.}
\begin{thm}
Let 
\[
F^{\mathrm{p,v}}(P,V;x)\coloneqq\sum_{n=0}^{\infty}c_{n}x^{n}
\]
where $c_{n}$ is the number of Motzkin paths of length $n$ with
every peak occuring at a height in $P\subseteq\mathbb{N}$ and every
valley occuring at a height in $V\subseteq\mathbb{P}$. Then, 
\[
F^{\mathrm{p,v}}(P,V;x)=\frac{1}{1-x+C_{1,0}-}\:\frac{(x+C_{3,0})^{2}}{1-x+C_{1,1}+C_{2,1}-}\:\frac{(x+C_{3,1})^{2}}{1-x+C_{1,2}+C_{2,2}-}\cdots
\]
where 
\[
\begin{array}{ccc}
C_{1,i}=\begin{cases}
\frac{x^{2}}{1-x^{2}}, & \mbox{if }i\notin P\mbox{ and }i+1\notin V\\
x^{2}, & \mbox{if }i\notin P\mbox{ and }i+1\in V\\
0, & \mbox{otherwise},
\end{cases} & \quad & C_{2,i}=\begin{cases}
\frac{x^{2}}{1-x^{2}}, & \mbox{if }i\notin V\mbox{ and }i-1\notin P\\
x^{2}, & \mbox{if }i\notin V\mbox{ and }i-1\in P\\
0, & \mbox{otherwise},
\end{cases}\end{array}
\]
and 
\[
C_{3,i}=\begin{cases}
\frac{x^{3}}{1-x^{2}}, & \mbox{if }i\notin P\mbox{ and }i+1\notin V\\
0, & \mbox{otherwise}.
\end{cases}
\]
\end{thm}
\begin{proof}
We weight both $UD$ and $DU$ by $t$, but we only wish to consider
instances of $UD$ at heights $i\notin P$ and instances of $DU$
at heights $i\notin V$. We claim that the cluster matrix is 
\[
\overrightarrow{L}_{\negmedspace G}(t)=\begin{bmatrix}\hat{C}_{1,0} & \hat{C}_{3,0}\\
\hat{C}_{4,1} & \hat{C}_{1,1}+\hat{C}_{2,1} & \hat{C}_{3,1}\\
 & \hat{C}_{4,2} & \hat{C}_{1,2}+\hat{C}_{2,2} & \ddots\\
 &  & \ddots & \ddots & \ddots\\
 &  &  & \ddots & \hat{C}_{1,m-1}+\hat{C}_{2,m-1} & \hat{C}_{3,m-1}\\
 &  &  &  & \hat{C}_{4,m} & \hat{C}_{2,m}
\end{bmatrix}
\]
where

\[
\begin{array}{ccc}
\hat{C}_{1,i}=\begin{cases}
\frac{UD(t-1)}{1-UD(t-1)^{2}}, & \mbox{if }i\notin P\mbox{, }i+1\notin V\\
UD(t-1), & \mbox{if }i\notin P\mbox{, }i+1\in V\\
0, & \mbox{otherwise},
\end{cases} & \quad & \hat{C}_{2,i}=\begin{cases}
\frac{DU(t-1)}{1-UD(t-1)^{2}}, & \mbox{if }i\notin V\mbox{, }i-1\notin P\\
DU(t-1), & \mbox{if }i\notin V\mbox{, }i-1\in P\\
0, & \mbox{otherwise},
\end{cases}\\
\vphantom{\frac{dy}{dx}} &  & \vphantom{\frac{dy}{dx}}\\
\hat{C}_{3,i}=\begin{cases}
\frac{UDU(t-1)}{1-DU(t-1)^{2}}, & \mbox{if }i\notin P\mbox{, }i+1\notin V\\
0, & \mbox{otherwise},
\end{cases} & \quad & \hat{C}_{4,i}=\begin{cases}
\frac{DUD(t-1)}{1-UD(t-1)^{2}}, & \mbox{if }i\notin V\mbox{, }i-1\notin P\\
0, & \mbox{otherwise}.
\end{cases}
\end{array}
\]

For example, $\hat{C}_{1,i}$ gives clusters starting and ending at
height $i$ and beginning with an up step. Every such cluster begins
with a peak, so if $i\in P$, then $\hat{C}_{1,i}=0$. Otherwise,
$i\notin P$, and if $i+1\in V$, then the only possible such cluster
is $UD$ because all other possible clusters begin with $UD$ and
are followed by a valley at height $i+1$. However, if $i\notin P\mbox{ and }i+1\notin V$,
then every subword of the form $UDUD\cdots$ is a valid cluster. One
can verify the formulas for $\hat{C}_{2,i},\hat{C}_{3,i},\hat{C}_{4,i}$
using similar reasoning, and the result follows from the same process
as before.
\end{proof}
Below are the generating functions for Motzkin paths with parity restrictions
on the heights of peaks and valleys:

\begin{align*}
F^{\mathrm{p,v}}(\mathbb{O},\mathbb{E}_{\geq0};x) & =\frac{1}{1-x+\frac{x^{2}}{1-x^{2}}-}\:\frac{(x+\frac{x^{3}}{1-x^{2}})^{2}}{1-x+\frac{x^{2}}{1-x^{2}}-}\:\frac{x^{2}}{1-x+\frac{x^{2}}{1-x^{2}}-}\\
 & \qquad\qquad\qquad\qquad\qquad\qquad\qquad\frac{(x+\frac{x^{3}}{1-x^{2}})^{2}}{1-x+\frac{x^{2}}{1-x^{2}}-}\:\frac{x^{2}}{1-x+\frac{x^{2}}{1-x^{2}}-}\cdots\\
 & =\frac{1-2x+2x^{2}-2x^{4}-\sqrt{1-4x+4x^{2}-4x^{4}}}{2x^{2}(1-x+x^{3})}\\
 & =1+x+x^{2}+2x^{3}+5x^{4}+12x^{5}+27x^{6}+60x^{7}+136x^{8}+\cdots
\end{align*}
\begin{align*}
F^{\mathrm{p,v}}(\mathbb{E}_{\geq0},\mathbb{O};x) & =\frac{1}{1-x-}\:\frac{x^{2}}{1-x+\frac{x^{2}}{1-x^{2}}-}\:\frac{(x+\frac{x^{3}}{1-x^{2}})^{2}}{1-x+\frac{x^{2}}{1-x^{2}}-}\\
 & \qquad\qquad\qquad\qquad\qquad\qquad\qquad\frac{x^{2}}{1-x+\frac{x^{2}}{1-x^{2}}-}\:\frac{(x+\frac{x^{3}}{1-x^{2}})^{2}}{1-x+\frac{x^{2}}{1-x^{2}}-}\cdots\\
 & =\frac{2(1-x+x^{3})}{1-2x+2x^{3}+\sqrt{1-4x+4x^{2}-4x^{4}}}\\
 & =1+x+2x^{2}+4x^{3}+8x^{4}+17x^{5}+38x^{6}+88x^{7}+208x^{8}+\cdots
\end{align*}

\begin{align*}
F^{\mathrm{p,v}}(\mathbb{O},\mathbb{O};x) & =\frac{1}{1-x+x^{2}-}\:\frac{x^{2}}{1-x-}\:\frac{x^{2}}{1-x+2x^{2}-}\:\frac{x^{2}}{1-x-}\:\frac{x^{2}}{1-x+2x^{2}-}\cdots\\
 & =\frac{2(1-x)}{1-2x+x^{2}+\sqrt{1-4x+6x^{2}-8x^{3}+5x^{4}-4x^{5}+4x^{6}}}\\
 & =1+x+x^{2}+2x^{3}+5x^{4}+12x^{5}+27x^{6}+60x^{7}+137x^{8}+\cdots
\end{align*}

\begin{align*}
F^{\mathrm{p,v}}(\mathbb{E}_{\geq0},\mathbb{E}_{\geq0};x) & =\frac{1}{1-x-}\:\frac{x^{2}}{1-x+2x^{2}-}\:\frac{x^{2}}{1-x-}\:\frac{x^{2}}{1-x+2x^{2}-}\:\frac{x^{2}}{1-x-}\cdots\\
 & =\frac{1-2x+3x^{2}-2x^{3}-\sqrt{1-4x+6x^{2}-8x^{3}+5x^{4}-4x^{5}+4x^{6}}}{2x^{2}(1-x)}\\
 & =1+x+2x^{2}+4x^{3}+7x^{4}+13x^{5}+27x^{6}+59x^{7}+131x^{8}+\cdots
\end{align*}
We note that the list of coefficients of $F^{\mathrm{p,v}}(\mathbb{E}_{\geq0},\mathbb{O};x)$
in particular is a shifted version of the OEIS sequence \cite[A025276]{oeis},
which can be verified by comparing generating functions.\\

\noindent \textbf{Acknowledgements.} The author thanks Ira Gessel
and Jordan Tirrell for reading earlier versions of the manuscript
and providing helpful suggestions; Cyril Banderier for several generous
discussions pertaining to this project at the 8th International Conference
on Lattice Path Combinatorics \& Applications; and an anonymous referee
for their constructive comments and suggestions.

\noindent \bibliographystyle{plain}
\addcontentsline{toc}{section}{\refname}\bibliography{bibliography}

\end{document}